\documentclass[12pt]{dalthesis}

\usepackage{tikz}
\usepackage{hexboard}
\usepackage{pgffor}
\usepackage{pspicture}
\usepackage{pstricks}
\usepackage{amsthm}
\usepackage{graphicx}
\usepackage{circuitikz}
\usepackage{amsthm}

\theoremstyle{plain}
\newtheorem{theorem}{Theorem}[chapter]
\newtheorem{lemma}[theorem]{Lemma}
\newtheorem{corollary}[theorem]{Corollary}

\theoremstyle{definition}
\newtheorem{definition}[theorem]{Definition}
\newtheorem{example}[theorem]{Example}

\newcommand{\hs}{\hspace{0.1cm}\hat{+}\hspace{0.1cm}}
\newcommand{\hleq}{\hspace{0.1cm} \hat{\leq} \hspace{0.1cm}}
\newcommand{\htri}{\hspace{0.1cm}\hat{\triangleleft}\hspace{0.1cm}}
\newcommand{\heq}{\hspace{0.1cm}\hat{\simeq}\hspace{0.1cm}}

\newcommand{\stard}{*_d}
\newcommand{\A}{\cal{A}}
\newcommand{\B}{\cal{B}}
\newcommand{\C}{\cal{C}}
\newcommand{\D}{\cal{D}}
\newcommand{\Le}{\cal{L}}
\newcommand{\R}{\cal{R}}
\newcommand{\N}{\cal{N}}

\def\gg#1|#2\gg{\{#1\mid#2\}}

\usepackage{graphicx}
\usepackage[hidelinks]{hyperref}
\hypersetup{
  colorlinks=true,
  linkcolor=blue,
  citecolor=black,
  filecolor=magenta,
  urlcolor=blue, }

\begin{document}

\title{The Combinatorial Game Theory of Rex+}
\author{Veronika Keras}

\bcshon  

\degree{Bachelor's of Science}
\degreeinitial{B. Sc}
\faculty{Science}
\dept{Department of Mathematics}

\defencemonth{April}\defenceyear{2026}

\nolistoftables
\nolistoffigures

\frontmatter

\begin{abstract}
In this thesis we present the combinatorial game theory of the game Rex+. Rex+ is a variant of the game Hex, played on a four sided board made out of hexagons. Both players take turns placing as many stones of their colour as they would like on the board, with the objective being to force the other player to connect their two sides. We describe a new ordering, and present some preliminary results on it.
\end{abstract}

\begin{acknowledgements}
Thanks to my friends and family for their support, and to my supervisor Peter Selinger for his guidance.
\end{acknowledgements}

\mainmatter

\chapter{Introduction}

Combinatorial game theory is the study of combinatorial games: games that are played between two players making alternating moves, with no random chance and perfect information.  Perfect information means that both players know everything about the game state.

Hex is a game played on a four sided board made out of hexagons. It is played between two players, Black and White, who take turns placing stones on the board with the aim of connecting their two sides of the board. Reverse Hex, or Rex, is a variant of Hex where the goal of the game is the opposite, to not connect your two sides, or equivalently to force your opponent to connect their two sides. Rex+ is a variant of Rex where each player can play more than one stone on their move. Selinger and Hockaday have analyzed Hex~\cite{hex_canon} and Rex~\cite{rex_cgt} respectively using the techniques of combinatorial game theory. In this thesis we will analyze Rex+ combinatorially.

First we will provide some background on important concepts in combinatorial game theory that we will use throughout this thesis. Then we will present some basic properties and examples to highlight why Rex+ is an interesting game to study. In Chapter~\ref{chap:CGT-rex+} we will present a new order and prove some preliminary results about it. We will then in Chapter~\ref{chap:weaker-option-closure} present the canonical form problem, and some weaker versions of option closure in an attempt to solve this problem. In Chapter~\ref{chap:canonical-forms} we will prove some results about existence and uniqueness of canonical forms under this order. Finally, in Chapter~\ref{chap:fallow-and-taut} we define a new property of some Rex+ positions, called fallowness, and present some alternative ways to characterize this property.

\chapter{Background}\label{chap:background}
	\section{Combinatorial Game Theory}
	
	We briefly recall some basic definitions from combinatorial game theory. For a more detailed introduction see ``On Numbers and Games"~\cite{ONAG}, ``Winning Ways"~\cite{WinningWays}, and ``Lessons in Play"~\cite{LessonsInPlay}.
	
	\subsection{Game Forms and Outcome Classes}
	We start by defining what a combinatorial game is. We say a game is combinatorial if it is played by two players who make alternating moves, with no random chance and perfect information. We usually call the two players Left and Right. 
	
	Some examples of combinatorial games are Tic Tac Toe, Chess, and Checkers. There are many different types of combinatorial games, which can classified by how the winner is determined. The most common type of combinatorial game is a class called normal play games. These are games where the objective is to be the last person to play, i.e. the first player who has no more moves loses.
	
	Some examples of normal play games include the games Nim and Domineering. 
	
	\begin{example}
		Nim is a two-player game played on a set of heaps. Each heap has a certain number of sticks in it. On their turn a player is allowed to remove as many sticks as they would like from one heap. The last player to take a stick wins.
	\end{example}
	
	\begin{example}
		Domineering is a two player game played on a grid. The two players Left and Right take turns playing a two-by-two domino on the grid, with Left playing vertical dominoes and Right playing horizontal dominoes.
	\end{example}
	
	Figure~\ref{fig:nim-example-position-1} shows an example of a possible Nim position, and Figure~\ref{fig:nim-example-position-2} is a legal move from the position in Figure~\ref{fig:nim-example-position-1}. Figure~\ref{fig:domineering-example-position} shows a sample position in Domineering, with one Left move and one Right move having been played.
	
	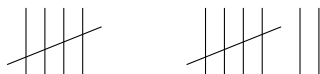
\begin{figure}
		\centering
		\begin{circuitikz}[baseline = (current bounding box.center)]
			\tikzstyle{every node}=[font=\fontsize{18.2pt}{23.7pt}\selectfont]
			\draw [short] (3.75,9.75) -- (3.75,8.875);
			\draw [short] (4,9.75) -- (4,8.875);
			\draw [short] (4.25,9.75) -- (4.25,8.875);
			\draw [short] (4.5,9.75) -- (4.5,8.875);
			\draw [short] (3.5,9) -- (4.75,9.5);
			\draw [short] (6.125,9.75) -- (6.125,8.875);
			\draw [short] (6.375,9.75) -- (6.375,8.875);
			\draw [short] (6.625,9.75) -- (6.625,8.875);
			\draw [short] (6.875,9.75) -- (6.875,8.875);
			\draw [short] (5.875,9) -- (7.125,9.5);
			\draw [short] (7.375,9.75) -- (7.375,8.875);
			\draw [short] (7.625,9.75) -- (7.625,8.875);
		\end{circuitikz}
		
		\caption{A Nim position with heaps of size 5 and 7.}
		\label{fig:nim-example-position-1}
	\end{figure}
	
	\begin{figure}
		\centering
		\begin{circuitikz}[baseline = (current bounding box.center)]
			\tikzstyle{every node}=[font=\fontsize{18.2pt}{23.7pt}\selectfont]
			\draw [short] (3.75,9.75) -- (3.75,8.875);
			\draw [short] (4,9.75) -- (4,8.875);
			\draw [short] (4.25,9.75) -- (4.25,8.875);
			\draw [short] (4.5,9.75) -- (4.5,8.875);
			\draw [short] (3.5,9) -- (4.75,9.5);
			\draw [short] (6.125,9.75) -- (6.125,8.875);
			\draw [short] (6.375,9.75) -- (6.375,8.875);
			\draw [short] (6.625,9.75) -- (6.625,8.875);
		\end{circuitikz}
		
		\caption{A Nim position with heaps of size 5 and 3.}
		\label{fig:nim-example-position-2}
	\end{figure}
	
	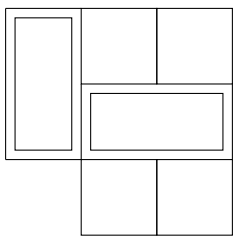
\begin{figure}
		\centering

		\begin{circuitikz}
			\tikzstyle{every node}=[font=\fontsize{18.2pt}{23.7pt}\selectfont]
			\draw  (2.625,12.125) rectangle (3.625,11.125);
			\draw  (3.625,12.125) rectangle (4.625,11.125);
			\draw  (2.625,10.125) rectangle (3.625,9.125);
			\draw  (3.625,10.125) rectangle (4.625,9.125);
			\draw  (1.75,12) rectangle (2.5,10.25);
			\draw  (2.75,11) rectangle (4.5,10.25);
			\draw [short] (2.625,12.125) -- (1.625,12.125);
			\draw [short] (1.625,12.125) -- (1.625,10.125);
			\draw [short] (1.625,10.125) -- (2.625,10.125);
			\draw [short] (4.625,11.125) -- (4.625,10.125);
			\draw [short] (2.625,11.125) -- (2.625,10.125);
		\end{circuitikz}

		\caption{A Domineering position with one move from Left and Right.}
		\label{fig:domineering-example-position}
	\end{figure}
	
	When we study games, we often write them out in what is called a game form. This is a way of recursively representing a game in terms of its options.
	
	\begin{definition}
		Given a game $G$, we say a \emph{left option} of $G$ is a position that the left player can make a move to. \emph{Right options} of $G$ are defined analogously.
	\end{definition}
	
	\begin{definition}
		Given a game $G$, we write $G = \gg G^{\cal{L}}|G^{\cal{R}}\gg$, where $G^{\cal{L}}$ is the set of left options, and $G^{\cal{R}}$ the set of right options.
	\end{definition}
	
	\begin{definition}
		For a typical left option of a game $G$ we write $G^L$, and for a typical right option we write $G^R$.
	\end{definition}
	
	\begin{definition}
		The game where neither player has any moves is written $0 = \gg|\gg$.
	\end{definition}	
	
	The game shown in Figure~\ref{fig:*-position} has the game form $\gg0|0\gg$ as both players have one option, namely to take the last stick, after which neither player has any remaining moves, so we are in the game $0$. This position has a special name and is called $*$. Figure~\ref{fig:1-position} shows a domineering position where Left has one move, and Right has no moves. It has the game form $\{0 \mid \}$. This position is given the name $1$.
	
	\begin{figure}
		\centering
		\begin{circuitikz}
			\tikzstyle{every node}=[font=\fontsize{18.2pt}{23.7pt}\selectfont]
			\draw [short] (3.75,9.75) -- (3.75,8.875);
		\end{circuitikz}

		\caption{The position $*$.}
		\label{fig:*-position}
	\end{figure}
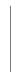
	
	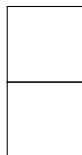
\begin{figure}
		\centering
		\begin{circuitikz}
			\tikzstyle{every node}=[font=\fontsize{18.2pt}{23.7pt}\selectfont]
			\draw  (1.625,12.125) rectangle (2.625,11.125);				\draw  (1.625,11.125) rectangle (2.625,10.125);
		\end{circuitikz}
		\caption{The position $1$}
		\label{fig:1-position}
	\end{figure}
	
	We now define the idea of an outcome class, a way of describing which player has a winning strategy in a given position. But before we do so we present an important observation on games, known as the fundamental theorem of combinatorial game theory.
	
	\begin{theorem}[The Fundamental Theorem of Combinatorial Game Theory]
		In a game $G$ where Alice plays first and Bob plays second, either Alice has a winning strategy or Bob has a winning strategy, but not both.
		\label{thm:fund-thm-of-cgt}
	\end{theorem}
	
	\begin{definition}\label{def:outcome-classes}
		We define the following \emph{outcome classes}:
		\begin{itemize}
			\item $\cal{N}$: Both players have first player winning strategies.
			\item $\cal{P}$: Neither player has a first player winning strategy.
			\item $\cal{L}$: Left has a first and second player winning strategy.
			\item $\cal{R}$: Right has a first and second player winning strategy.
		\end{itemize}
		\hspace{0.1cm}
		We write $o(G)$ to denote the outcome class of a game $G$.
	\end{definition}
	
	Theorem~\ref{thm:fund-thm-of-cgt} tells us that those are the only possible outcome classes that can happen.
	
	Finally, we define a partial order on the outcome classes. This partial order is defined based on which outcome class is better for Left. Figure~\ref{fig:outcome-class-poset} shows the partial order for outcome classes.
	
	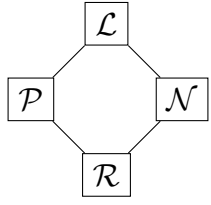
\begin{figure}
		\[
		\begin{tikzpicture}[baseline=(current bounding box.center)]
			\node[draw] (L) at (2,1) {$\Le$};
			\node[draw] (N) at (1,0) {$\cal{P}$};
			\node[draw] (P) at (3,0) {$\cal{N}$};
			\node[draw] (R) at (2,-1) {$\R$};
			\draw (L) -- (P);
			\draw (L) -- (N);
			\draw (P) -- (R);
			\draw (N) -- (R);
		\end{tikzpicture}
		\]
		\caption{The partial order on outcome classes.}
		\label{fig:outcome-class-poset}
	\end{figure}
	
	\subsection{Sum and Order}
	
	In combinatorial game theory, a useful technique for analyzing games is to divide the games into smaller parts and analyze them separately. To do this it is useful to define a way to add games together. Consider two games placed side by side, where you can play in one game or the other but not both. This is the idea that sums of games attempt to capture.
	
	\begin{definition}
		The game $G + H$ is defined recursively as $G + H = \gg G^{L}+H, G+H^{L}|G^{R}+H, G+H^{R}\gg$.
	\end{definition}
	
	Next we present the contextual order on normal play games, a way to compare two games to see which is better. The contextual order takes two games and says that one is better than the other if it is better for Left in every context, or more formally:
	
	\begin{definition}
		We say a game $G \leq_C H$ if for all games $X$, we have that $o(G+X) \leq o(H+X)$.
	\end{definition}
	
	Finally we define an intrinsic order on normal play games, which is a way to compare two games without having to consider every possible context $X$.
	
	\begin{definition}
		$G \leq H$ if 
		\begin{itemize}
			\item[1)] $\forall G^L, G^L \triangleleft H$ and;
			\item[2)] $\forall H^R, G \triangleleft H^R$.
		\end{itemize}
		$G \triangleleft H$ if 
		\begin{enumerate}
			\item[1)] $\exists G^R, G^R \leq H$ or;
			\item[2)] $\exists H^L, G \leq H^L$.
		\end{enumerate}
	\end{definition}
	
	\begin{definition}
		We say that $G$ and $H$ are \emph{equivalent}, written $G \simeq H$, if $G \leq H$ and $H \leq G$.
	\end{definition}
	
	We claim without proof that this intrinsic definition is equivalent to the contextual one. For more details see ``Lessons in Play"~\cite{LessonsInPlay}.
	
	\subsection{Canonical Forms}
	
	When we discuss games in terms of their game forms, the game forms become very large very quickly. To combat this, there are some reductions that can be applied to game forms to simplify them while preserving equivalence.
	
	The first is a reduction called domination. We present all these results without proof, for more details see ``Lessons in Play"~\cite{LessonsInPlay}.
	
	\begin{definition}
		A left option $A$ of a game $G$ is called \emph{dominated} if there is another left option $B$ such that $A \leq B$. We say $A$ is \emph{dominated by} $B$. The corresponding notion for a right option is defined dually.
	\end{definition}
	
	\begin{theorem}
		If $G = \gg A, B, G^{\Le}|G^{R}\gg$ and $A$ is dominated by $B$ then $G \simeq G'$, where $G' = \gg B, G^{\Le}|G^{R}\gg$.
	\end{theorem}
	
	The second is a reduction called reversibility.
	
	\begin{definition}
		A left option $A$ of a game $G$ is called \emph{reversible} if it has a right option $A^R$ such that $A^R \leq G$. We say $A$ is \emph{reversible via} $A^R$. The corresponding notion is defined analogously for right options.
	\end{definition}
	
	\begin{theorem}
		If $G = \gg A, G^{\Le}|G^{\R}\gg$ and $A$ is reversible by $A^R$ then $G \simeq G'$, where $G' = \gg A^{R\Le}, G^{\Le}|G^{\R}\gg$
	\end{theorem}

	\section{Hex, Rex, and Rex+: the Rules}
	
	Hex is a game first developed by Piet Hein in 1942~\cite{Hein}. It is played on a hexagonal board, like in Figure~\ref{fig:hex-board}. The two players Black and White take turns placing stones of their own colour on the board, with the goal of connecting the two sides of their colour. 
	
	\begin{figure}
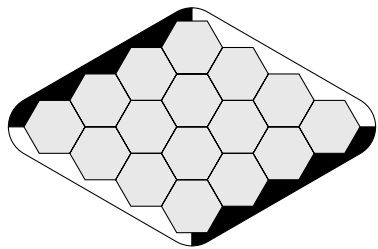

		\[
		\begin{hexboard}
			\board(4,4)
		\end{hexboard}
		\]
		\caption{A 4x4 hexboard.}
		\label{fig:hex-board}
	\end{figure}
	
	Figure~\ref{fig:black-win-hex} shows a winning position for Black, and Figure~\ref{fig:white-win-hex} shows a winning position for White. Hex was analyzed using the methods of combinatorial game theory by Selinger in~\cite{hex_canon}.
	
	\begin{figure}
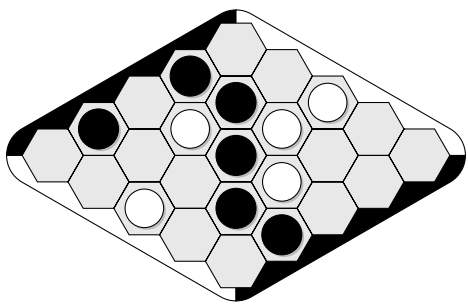

		\[
		\begin{hexboard}
			\board(5,5)
			\black(3,3)
			\white(1,3)
			\black(2,1)
			\black(4,1)
			\black(4,2)
			\black(2,4)
			\black(2,5)
			\white(3,2)
			\white(4,3)
			\white(5,3)
			\white(3,4)
		\end{hexboard}
		\]
		\caption{A winning position in Hex for Black.}
		\label{fig:black-win-hex}
	\end{figure}
	
	\begin{figure}
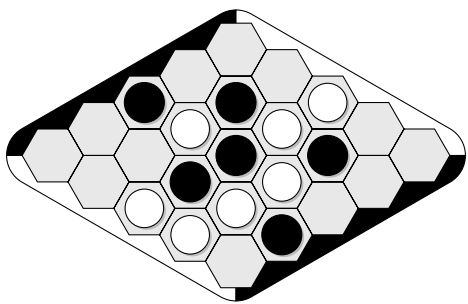

		\[
		\begin{hexboard}
			\board(5,5)
			\black(3,3)
			\white(1,3)
			\black(3,1)
			\black(4,2)
			\white(2,4)
			\black(2,5)
			\white(3,2)
			\white(4,3)
			\white(5,3)
			\white(3,4)
			\white(1,4)
			\black(2,3)
			\black(4,4)
		\end{hexboard}
		\]
		\caption{A winning position in Hex for White.}
		\label{fig:white-win-hex}
	\end{figure}
	
	An important concept in Hex and related games is that of a dead cell which we now define.
	
	\begin{definition}
		Consider a position. A cell $x$ is called \emph{dead} if for every way of filling the position with stones, the outcome of the game does not change whether or not $x$ is empty, filled by a black stone, or filled by a white stone. We will sometimes write $*_d$ for a dead cell.
	\end{definition}
	
	Figure~\ref{fig:dead-cell} shows an example of a dead cell, marked with an $x$.
	
	\begin{figure}
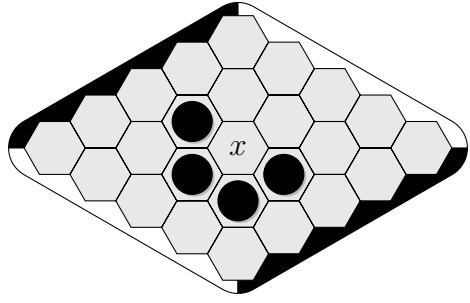

		\[
		\begin{hexboard}
			\board(5,5)
			
			\black(2,4)
			\black(2,3)
			\black(3,4)
			\black(3,2)
			\hex(3,3)\label{$x$}
		\end{hexboard}
		\]
		\caption{A dead cell in Hex.}
		\label{fig:dead-cell}
	\end{figure}
	
	Rex, which stands for reverse Hex, is a variant of Hex where the goal of the game for both players is not to connect. Figure~\ref{fig:black-win-rex} shows a Rex position where Black has won. Rex was analyzed combinatorially by Hockaday in~\cite{rex_cgt}.
	
	\begin{figure}
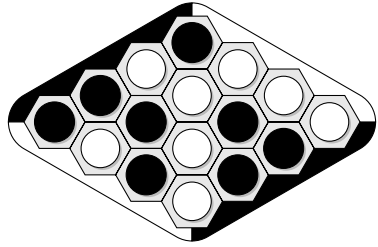

		\[
		\begin{hexboard}
			\board(4,4)
			
			\black(1,1)
			\black(2,1)
			\black(4,1)
			\black(1,3)
			\black(2,4)
			\black(3,3)
			\black(3,4)
			\black(2,2)
			
			\white(1,2)
			\white(1,4)
			\white(2,3)
			\white(3,2)
			\white(4,2)
			\white(4,3)
			\white(4,4)
			\white(3,1)
		\end{hexboard}
		\]
		\caption{A winning position for Black in Rex.}
		\label{fig:black-win-rex}
	\end{figure}
	
	Rex+ is a variant of Rex where each player is allowed to play one or more stones of their colour on their turn, unlike Rex where they may only play one stone at a time. In this thesis we will analyze Rex+ using the strategies of combinatorial game theory.

	\section{The Knaster-Tarski Fixed Point Theorem}
	
	An important theorem we will use later in Chapter~\ref{chap:weaker-option-closure} is the Knaster-Tarski Fixed Point Theorem. This theorem tells us that every monotone function on a power set has a least fixed point. There is a more general version of this theorem for monotone functions on complete lattices, but here we only present the proof for power sets. For more, see~\cite{tarski1955lattice}.
	
	\begin{definition}
		A function $\phi$ on a partially ordered set is \emph{monotone} if it preserves the partial order, i.e., $x \leq y \Rightarrow \phi(x) \leq \phi(y)$.
	\end{definition}
	
	\begin{definition}
		Given a power set $\cal{P}(S)$, and $X, Y \in \cal{P}(S)$, we write $X \leq Y$ if $X \subseteq Y$.
	\end{definition}
	
	\begin{theorem}\label{thm:knaster-tarski}
		Given a set S, and a monotone function $\phi: \cal{P}(S) \rightarrow \cal{P}(S)$, $\phi$ has a least fixed point.
	\end{theorem}
	
	\begin{proof}
		We define a set $A$ where $A = \{X \in {\cal P}({\cal S}) \mid \phi(X) \leq X\}$.
		
		Then $A$ is non-empty since the set $S \in \cal{P}(S)$ is in $A$. So $A$ has an infimum, namely, the intersection of all $X \in A$. Let us call it $Y$. Then $Y \leq X$ for every $X$ in $A$. Thus $\phi(Y) \leq \phi(X)$ since $\phi$ is monotone. Furthermore, $\phi(X) \leq X$. So $\phi(Y) \leq X$.
		
		This tells us that $\phi(Y)$ is a lower bound for $A$. Since $Y$ is the infimum of $A$ we have that $\phi(Y) \leq Y$. Thus, since $\phi$ is monotone it follows that $\phi(\phi(Y)) \leq \phi(Y)$.
		
		It follows that $\phi(Y)$ is in $A$. So $Y \leq\phi(Y)$. So we have $Y \leq \phi(Y) \leq Y$. Thus $Y$ is a fixed point. Furthermore, $Y$ is the smallest fixed point, since every fixed point is in $A$ and $Y$ is less than or equal to every element in $A$.
	\end{proof}

\chapter{Rex+ by Example}\label{chap:examples}	
	In this chapter we will informally cover some interesting properties and positions of Rex+. In later chapters we will cover much of this more formally.
	
	\section{$\cal{N}$ Positions}
	
	We saw in chapter~\ref{chap:background} the concept of an $\cal{N}$ position, a position where both players have a first player winning strategy. Surprisingly, we find that given an $\cal{N}$ position in Rex+ both players have the same winning set of cells, and that it is unique for both.
	
	\begin{theorem}
		If a Rex+ position $G$ is an $\cal{N}$ position, then there is a unique winning move for both players, and it is the same.
	\end{theorem}
	
	\begin{proof}
		Let the set $X$ represent a set of cells that is a winning move for Black, and $Y$ represent a set of cells that is a winning move for White. Let $Z$ be the intersection of $X$ and $Y$. 
		
		Suppose $X \neq Y$. Then, there are without loss of generality some stones in $Y$ that are not in $X$. Consider the game where Black plays first and plays $X$. Black is now winning this game, since $X$ is a winning move. Suppose White now plays the stones that are in $Y$ but not in $X$, then Black has a winning response, $W$. So Black is winning in the position where it is White's turn, White has played in $Y \setminus X$, and Black has played in $X$ and $W$. Consider instead the game where White plays first and plays $Y$, and Black responds with $W$ and any stones which are in $X$ but not $Y$ (of which there could be none). White is winning in this position since $Y$ was a winning first move. However, the only difference between this position and the previous one, is that the stones in $Z$ have gone from black to white. Having black stones become white is only ever good for Black, so if Black was winning in the previous position, they ought to still be winning. But we said White was winning here, since $Y$ is a winning first move. This is a contradiction, so we cannot have $Z$ is non empty, and $X \neq Y$.
		
		Thus we must have that $X = Y$. Furthermore, if we have two winning moves for Black, $X$ and $W$, which are not equal, then $W$ must also be equal to $Y$ by the above argument, so we would have winning moves for Black and White that are not equal, which we have just proved is not possible. So there is a unique winning move for both Black and White, and it is the same move.
	\end{proof}
	
	Figure~\ref{fig:N-position-dead-cell} shows an example of an $\cal{N}$ position in Rex+. We note that this position contains two dead cells, marked with $x$'s, and that playing in those dead cells is the winning move for both players. We find more generally, that any time there is a set of dead cells in a Rex+ position that the player whose turn it is always wants to fill all of the dead cells. The idea behind why this is true, is that if you don't, then you have only given your opponent more options. If they had a move they wanted to make in the non dead cells, then they can still make this move and also play in the dead cells, but if they had no move they wanted to make in the non dead cells, then they can simply play in the dead cells, and not the rest of the board, and since they are dead cells, they cannot have made the position worse than if you had played in the dead cells yourself.
	
	\begin{figure}
		\[
		\begin{hexboard}
			\board(3,3)
			
			\white(1,1)
			\white(2,1)
			\white(2,2)
			\white(1,3)
			
			\black(3,1)
			\black(3,3)
			
			\hex(1,2)\label{$x$}
			\hex(2,3)\label{$x$}
		\end{hexboard}
		\]
		\label{fig:N-position-dead-cell}
		\caption{An $\cal{N}$ position in Rex+.}
	\end{figure}
	
	From this it is clear to see that having many dead cells is the same as having one dead cell, as both players would just like to fill all of them anyways, and since this is Rex+ they can fill many dead cells in one turn as easily as they could fill one dead cell.
	
	One might think that the only way that an $\cal{N}$ position would arise is if there are dead cells on the board. We find that this is not the case. Figure~\ref{fig:N-position-not-dead} shows an example of an $\cal{N}$ position, with the winning set of moves marked with an $x$. However the cell $x$ is not in fact a dead cell.

	\begin{figure}
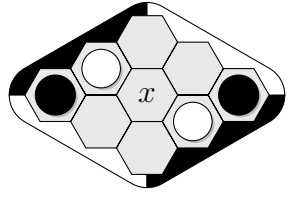

		\[
		\begin{hexboard}
			\board(3,3)
			
			\black(1,1)
			\black(3,3)
			
			\white(2,1)
			\white(2,3)
			
			\hex(2,2)\label{$x$}
		\end{hexboard}
		\]
		\caption{An $\cal{N}$ position in Rex+ with no dead cells.}
		\label{fig:N-position-not-dead}
	\end{figure}
	
	To conclude this section we remark that $\cal{N}$ positions do not appear frequently in games under optimal play, as they could only possibly be the starting position. No player ever plays to an $\cal{N}$ position under optimal play, as they would just play past it to a position where their opponent was losing instead.
	
	\section{$*$-Antimonotonicity}
	
	In studying Hex, one can remark that it is a strongly monotone game~\cite{hex_canon}, that is to say, in Hex one would always rather have a stone of their colour than an empty cell, and an empty cell than a stone of the opposite colour. One might expect that the opposite is true in Rex+ but we do not find this to be the case. We do find that it is always better to have a stone of the other player's colour than to have one of your own colour, since you can only lose possible ways to connect your sides by switching one of your stones to the opposite colour. However, sometimes it is indeed advantageous to have a stone of your own colour rather than an empty cell. For instance, Figure~\ref{fig:antimonotone-ex} shows two positions, where the only difference is that in the right one Black has played in cell $x$. In the right position, Black is winning, and in the left position Black is losing, as White can play in $x$. This demonstrates the idea that the only benefit one can gain from having one of their stones rather than an empty cell is by taking away a place to play from their opponent.
	
	\begin{figure}
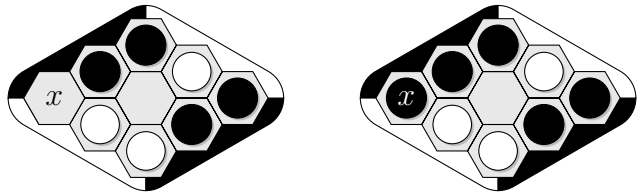

		\[
		\begin{hexboard}
			\board(3,3)
			
			\hex(1,1)\label{$x$}
			
			\black(2,1)
			\black(2,3)
			\black(3,3)
			\black(3,1)
			
			\white(1,2)
			\white(1,3)
			\white(3,2)
		\end{hexboard}
		\hspace{1cm}
		\begin{hexboard}
			\board(3,3)
			
			\black(1,1)\label{$x$}
			
			\black(2,1)
			\black(2,3)
			\black(3,3)
			\black(3,1)
			
			\white(1,2)
			\white(1,3)
			\white(3,2)
		\end{hexboard}
		\]
		\caption{Two Rex+ positions. Black is winning in the right position, and losing in the left position.}
		\label{fig:antimonotone-ex}
	\end{figure}
	
	Informally, we say that it is always better to have an empty cell than to play a move and add a dead cell, because that dead cell gets rid of any benefit by adding back a place for the opponent to play. This property is called $*$-antimonotonicity and will be expanded on in Chapter~\ref{chap:fallow-and-taut}.
	
	\section{Premotivity}	
	In this section we will establish a property called premotivity. First though, we must give some background.
	
	Suppose we have a game and its opposite laid out, i.e. a Rex+ position and the same position with all the colours swapped, see Figure~\ref{fig:premotivity-ex} for an example. Then imagine we play both games, with each player having the option to move in one game or the other, but not both. Then the second player has a strategy to win at least one of the games, called the copycat strategy. This strategy will be familiar to those who have studied combinatorial game theory, but the idea is as follows: whichever move the first player makes, the second player copies their move in the other game. Then at the end of the game both boards will be identical except that the colours will have been swapped. Then a different colour wins both games, so player two must have won one of the games. We call this a copycat strategy.
	
	\begin{figure}
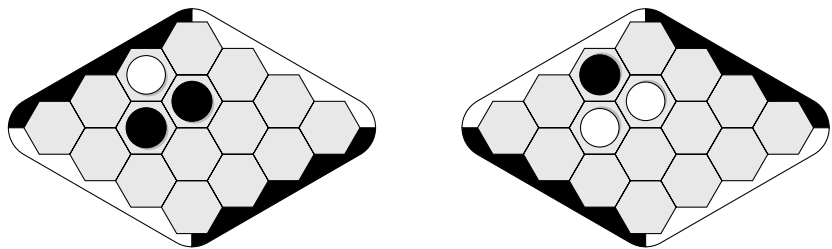

		\[
		\begin{hexboard}[baseline = (current bounding box.center)]
			\board(4,4)
			
			\black(3,2)
			\white(3,1)
			\black(2,2)
		\end{hexboard}
		\hspace{-3 cm}
		\scalebox{1}[-1]{
		\begin{hexboard}[baseline = (current bounding box.center)]
			\board(4,4)
			
			\black(1,3)
			\white(2,2)
			\white(2,3)
		\end{hexboard}
		}
		\]
		\caption{A Rex+ position and its opposite.}
		\label{fig:premotivity-ex}
	\end{figure}
	
	This second player copycat can be used for any game. What is special about Rex+ is that we can describe a first player copycat strategy as well. The strategy is as follows: we have the same set up, with two identical boards except that the colours of the stones and the edges have been swapped. Then the first player plays one stone anywhere on the board. After that they simply ignore the stone and follow the same copycat strategy as described until one of two things happens. If the second player plays a move of several stones which encompasses the first player's original move, then the first player copies that move except their original move. After that the two boards will be exactly the same, and the first player can employ the copycat strategy for the rest of the game. If the second player plays the first player's original stone, and only that one stone, then the first player can just do as they did at the start of the game and pick another random cell to play in. At the end of the game, it will be the second player's turn and they will be forced to play in the one remaining cell, the mirror of the move played by the first player, and make the two boards equal. Once the two boards are equal, the first player must again have won in at least one of the boards.
	
	This property of having a first player copycat strategy is called premotivity, and is one shared with ordinary set colouring games, like Rex as shown in Hockaday~\cite{rex_cgt}. However, that it is also possible for Rex+ is surprising.
	
\chapter{The Combinatorial Game Theory of Rex+}\label{chap:CGT-rex+}
	\section{Outcome Posets}
	
	Rex+ is what is called a game over a partially ordered set, or poset. This means its outcomes can be described in terms of a partially ordered set. 
	
	\begin{definition}
		A \emph{game over a partially ordered set} is a type of game where the final outcomes for any position form a partially ordered set, with the relation $a \leq b$ if and only if outcome $b$ is preferred by Left to outcome $a$. An element in the set of possible outcomes for a game is called an \emph{atom}.
	\end{definition}

	Consider Figure~\ref{fig:example-position}. There are five possible outcomes for the Left player. Either they connect all of their three terminals 1, 2 and 3, they connect terminals 1 and 2, they connect terminals 1 and 3, they connect terminals 2 and 3, or they connect none of their terminals. The option where Left connects none of their terminals is the best option for Left, so we call it top, written $\top$. The one where Left connects all three of their terminals is the worst option for Left, so we call it bottom, written $\bot$. The other three options are all incomparable to one another, as which one would be best depends on what is going on on the rest of the board. Then these outcomes form the partially ordered set seen in Figure~4.2.
	
	\begin{figure}
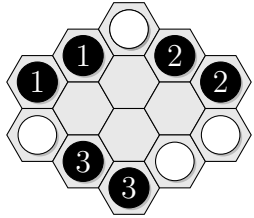

		\[
		\begin{hexboard}[baseline = (current bounding box.center)]
			\hex(2,1)
			\hex(3,1)
			\hex(4,1)
			\hex(1,2)
			\hex(2,2)
			\hex(3,2)
			\hex(4,2)
			\hex(1,3)
			\hex(2,3)
			\hex(3,3)
			\hex(4,3)
			\hex(1,4)
			\hex(2,4)
			\hex(3,4)
			
			\black(2,1)\label{1}
			\black(3,1)\label{1}
			\black(1,3)\label{3}
			\black(1,4)\label{3}
			\black(4,2)\label{2}
			\black(4,3)\label{2}
			
			\white(1,2)
			\white(4,1)
			\white(2,4)
			\white(3,4)
			\white(4,1)
		\end{hexboard}
		\]
		\caption{A sample Rex+ position.}
		\label{fig:example-position}
	\end{figure}
	
	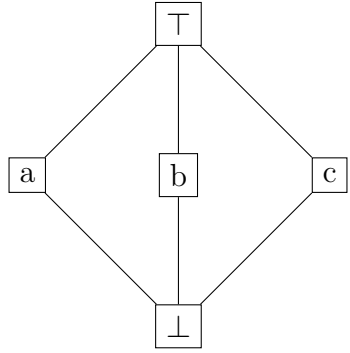
\begin{figure}
		\[
		\begin{tikzpicture}[baseline=(current bounding box.center)]
			\node[draw] (top) at (2,2) {$\top$};
			\node[draw] (a) at (0,0) {a};
			\node[draw] (b) at (2,0) {b};
			\node[draw] (c) at (4,0) {c};
			\node[draw] (bot) at (2,-2) {$\bot$};
			\draw (top) -- (a);
			\draw (top) -- (b);
			\draw (top) -- (c);
			\draw (a) -- (bot);
			\draw (b) -- (bot);
			\draw (c) -- (bot);
		\end{tikzpicture}
		\]
		\label{fig:poset}
		\caption{The partially ordered set representing the possible outcomes of the position in Figure~\ref{fig:example-position}.}
	\end{figure}
	
	\section{Contextual and Intrinsic Orders}
	
	We start by defining an important property that all Rex+ games have called option closure.
	
	\begin{definition}
		A game $G$ is called \emph{option closed} if every $G^{LL}$ is itself a $G^L$ and every $G^{RR}$ is itself a $G^R$ (and similarly for all options of $G$).
	\end{definition}
	
	\begin{lemma}
		Every Rex+ game is option closed.
	\end{lemma}
	
	\begin{proof}
		Consider a game $G$ which is a Rex+ position. Then consider any $G^{LL}$. It is achieved by Black playing a set of black stones $X$ and then a set of black stones $Y$. But Black could have played $X \cup Y$ in one move, so $G^{LL}$ is a left option of $G$. The same is true for Right.
	\end{proof}
	
	We recall from Chapter~\ref{chap:background} the idea of a sum, a way to combine a set of games into one larger game. Since Rex+ is an option closed game, we must use a new sum that preserves the option closed property. We use the notation $\hs$ to represent this sum.
	
	\begin{definition}
		Given games $G$ and $H$, we define $G\hs H$ as follows: $G \hs H = \gg G^{L}\hs H, G\hs H^{L}, G^{L}\hs H^{L}|G^{R}\hs H, G\hs H^{R}, G^{R}\hs H^{R}\gg$
	\end{definition}
	
	We now begin defining a contextual order for option closed games. This is defined analogously to the contextual order on normal play games, but using our option closed sum.
	
	\begin{definition}
		We say $G \leq_C H$ if for all games $X$, $o(G \hs X) \leq o(H \hs X)$.
	\end{definition} 
	
	We then also define a new intrinsic order on option closed games as follows.
	
	\begin{definition} \label{def:intrinsic-order}
		$G \hleq H$ if 
		\begin{itemize}
			\item [1.] $\forall G^L, G^L \htri H$ and;
			\item [2.] $\forall H^R, G \htri H^R$ and;
			\item [3.] if $G = [a], H = [b]$ then $a \leq b$.
		\end{itemize}
		\vspace{0.1cm}
		$G \htri H$ if 
		\begin{itemize}
			\item [1.] $\exists G^R, G^R \hleq H$ or;
			\item [2.] $\exists H^L, G \hleq H^L$ or;
			\item [3.] $\exists G^R, H^L,  G^R \hleq H^L$.
		\end{itemize}
	\end{definition}
	
	\begin{lemma} \label{thm:order-reflexivity}
		For all games $G$, $G\hleq G$.
	\end{lemma}
	
	\begin{proof}
		By induction. Take a left option $G^L$ of $G$. Then, by induction hypothesis, $G^L \hleq G^L$ $\Rightarrow G^L \htri G$.
		
		Next, take a right option $G^R$ of $G$. By induction hypothesis $G^R \hleq G^R$ $\Rightarrow G \htri G^R$.
		
		Finally, if $G$ = $[a]$ is atomic, then we have that $a$ $\leq$ $a$, so $[a]$ $\hleq$ $[a]$.
	\end{proof}
	
	\begin{lemma}\label{thm:order-sums}
		Let $G$, $H$, and $K$ be games. Then:
		\begin{itemize}
			\item [1)] If $G \hleq H$, then $G \hs K \hleq H \hs K$.
			\item [2)] If $G \htri H$, then $G \hs K \htri H \hs K$.
		\end{itemize}
	\end{lemma}
	
	\begin{proof}
		We prove this by mutual induction on the two statements.
		
		\begin{itemize}
			\item [1)] Suppose $G \hleq H$. We wish to show $G \hs K \hleq H \hs K$. First, take a left option $A$ of $G \hs K$. 
			\begin{itemize}
				\item [] Case i) $A = G^L \hs K$ for some left option $G^L$ of $G$. Then, since $G \hleq H$, $G^L \htri H$. 
				
				Thus, by the induction hypothesis $A = G^L \hs K \htri H \hs K$.
				
				\item [] Case ii) $A = G \hs K^L$ for some left option $K^L$ of $K$. Then, by the induction hypothesis $G \hs K^L \hleq H \hs K^L$. Thus, $A = G \hs K^L \htri H \hs K$.
				
				\item [] Case iii) $A = G^L \hs K^L$ for some left options $G^L, K^L$ of $G$ and $K$ respectively. Since $G \hleq H$, $G^L \htri H$. This results in three possibilities, so we do further case distinction.
				\begin{itemize}
					\item [] Case a) $G^{LR} \hleq H$. Then, by the induction hypothesis, we have that $G^{LR} \hs K^L \hleq H\hs K^L$. Thus, $A = G^L \hs K^L \htri H \hs K$.
					\item [] Case b) $G^L \hleq H^L$. Then by the induction hypothesis, we have that $G^{L} \hs K^L \hleq H^L \hs K^L$. Thus, $A = G^L \hs K^L \htri H \hs K$.
					\item [] Case c) $G^{LR} \hleq H^L$. Then, by the induction hypothesis, we have that $G^{LR} \hs K^L \hleq H^L\hs K^L$. Thus, $A = G^L \hs K^L \htri H \hs K$. So $A \htri H\hs K$.
				\end{itemize}
			\end{itemize}
			
			Now we take a right option $B$ of $H \hs K$. 
			\begin{itemize}
				\item [] Case i) $B = H^R \hs K$ for some right option $H^R$ of $H$. Then, since $G \hleq H$, $G \htri H^R$. Thus, by the induction hypothesis $G \hs K \htri H^R \hs K = B$.
				\item [] Case ii) $B = H \hs K^R$ for some right option $K^R$ of $K$. Then, by the induction hypothesis $G \hs K^R \hleq H \hs K^R$. Thus, $G \hs K \htri H \hs K^R = B$.
				\item [] Case iii) $B = H^R \hs K^R$ for some right options $H^R, K^R$ of $H$ and $K$ respectively. Then since $G \hleq H$, $G \htri H^R$. This results in three possibilities, so we do further case distinction.
				\begin{itemize}
					\item [] Case a) $G \hleq H^{RL}$. Then, by the induction hypothesis, we have that $G \hs K^R \hleq H^{RL}\hs K^R$. Thus, $G \hs K \htri H^R \hs K^R = B$.
					\item [] Case b) $G^R \hleq H^R$. Then by the induction hypothesis, we have that $G^R \hs K^R \hleq H^R \hs K^R$. Thus, $G \hs K \htri H^R \hs K^R = B$.
					\item [] Case c) $G^R \hleq H^{RL}$. Then, by the induction hypothesis, we have that $G^R \hs K^R \hleq H^{RL}\hs K^R$. Thus, $G \hs K \htri H^R \hs K^R = B$.
				\end{itemize}
			\end{itemize}
			
			Finally, if $G \hs K = [a]$ and $H \hs K = [b]$ are both atomic, we have that $G = [c]$, $H = [d]$ and $K = [e]$ are all atomic as well. Then, since $G \hleq H$, we have that $c \leq d$. Thus, $a = c + e \leq d + e = b$. So we must have $G \hs K = [a] \hleq [b] = H \hs K$.
			
			Thus, $G \hs K \hleq H \hs K$.
			
			\item [2)] Suppose $G \htri H$. We wish to show $G \hs K \htri H \hs K$. There are three ways we can have $G \htri H$.
			\begin{itemize}
				\item [] Case i) $G^R \hleq H$. Then, by the induction hypothesis, $G^R \hs K \hleq H \hs K$. Thus $G \hs K \htri H \hs K$. 
				\item [] Case ii) $G \hleq H^L$. Then, by the induction hypothesis, $G \hs K \hleq H^L \hs K$. Thus $G \hs K \htri H \hs K$. 
				\item [] Case iii) $G^R \hleq H^L$. Then, by the induction hypothesis, $G^R \hs K \hleq H^L \hs K$. Thus $G \hs K \htri H \hs K$. 
			\end{itemize}
			
			So  we have $G \hs K \htri H \hs K$.\qedhere
		\end{itemize}
	\end{proof}
	
	\begin{lemma}\label{thm:order-pre-outcome-classes}
		Given option closed games $G$ and $H$:
		\begin{itemize}
			\item [1)] If Left has a second player winning strategy in $G$ and $G \hleq H$, then Left has a second player winning strategy in $H$.
			\item [2)] If Left has a second player winning strategy in $G$ and $G \htri H$, then Left has a first player winning strategy in $H$.
			\item [3)] If Left has a first player winning strategy in $G$ and $G \hleq H$, then Left has a first player winning strategy in $H$.
		\end{itemize}
	\end{lemma}
	
	\begin{proof}
		We prove these statements by mutual induction.
		\begin{itemize}
			\item [1)] We show that Left has a first player winning strategy in every right option of $H$. Take $H^R$. We have that $G \htri H^R$, so by induction hypothesis (2) it follows that Left has a first player winning strategy in $H^R$. Thus, Left has a first player winning strategy in $H$.
			\item [2)] There are three ways we can have $G \htri H$.
			\begin{itemize}
				\item [] Case i) $G^R \hleq H$. Left has a first player winning strategy in $G^R$, since they have a second player winning strategy in $G$. Thus by induction hypothesis (3), Left has a first player winning strategy in $H$.
				\item [] Case ii) $G \hleq H^L$. By induction hypothesis (1), Left has a second player winning strategy in $H^L$. Thus, they have a first player winning strategy in $H$.
				\item [] Case iii) $G^R \hleq H^L$. Left has a first player winning strategy in $G^R$, since they have a second player winning strategy in $G$. Thus by induction hypothesis (3), Left has a first player winning strategy in $H^L$. Thus, there is some $H^{LL}$ where Left has a second player winning strategy. Since $H$ is option closed, $H^{LL}$ is a Left option of $H$. Thus, Left has a first player winning strategy in $H$.
			\end{itemize}
			\item [3)] Since Left has a first player winning strategy in $G$, there is some $G^L$ where Left has a second player winning strategy. $G^L \htri H$, so by induction hypothesis (2), it follows that Left has a first player winning strategy in $H$.\qedhere
		\end{itemize}
	\end{proof}
	
	\begin{corollary}\label{thm:order-outcome-classes}
		For option closed games $G$ and $H$, if $G \hleq H$ then $o(G) \leq o(H)$.
	\end{corollary}
	
	\begin{proof}
		This follows immediately from Lemma~\ref{thm:order-pre-outcome-classes}.
	\end{proof}
	
	\begin{lemma}
		For option closed games $G$ and $H$, if $G \hleq H$ then $G \leq_C H$.
	\end{lemma}
	
	\begin{proof}
		Suppose $G \hleq H$. Let $X$ be an arbitrary game. Then by Lemma~\ref{thm:order-sums} we have that $G \hs X \hleq H \hs X$. It follows from Lemma~\ref{thm:order-outcome-classes}, that $o(G \hs X) \leq o(H \hs X)$. So $G \leq_C H$.
	\end{proof}
	
	\begin{lemma}\label{thm:atoms-not-comp}
		Given a game $G$ and atoms $a$ and $b$, 
		\begin{itemize}
			\item [1)] $[a] \hleq G$ and $G \htri [b]$ is not possible.
			\item [2)] $[a] \htri G$ and $G \hleq [b]$ is not possible.
		\end{itemize}
	\end{lemma}
	
	\begin{proof}
		We prove this by mutual induction on the two statements.
		\begin{itemize}
			\item [1)] Suppose we had $[a] \hleq G$ and $G \htri [b]$. From the second part we must have $G^R \hleq [b]$ since $[b]$ has no left options. But $[a] \htri G^R$ by the first part. This is a contradiction by statement (2).
			\item [2)] Suppose we had $[a] \htri G$ and $G \hleq [b]$. From the first part we must have $[a] \hleq G^L$ since $[a]$ has no right options. But $G^L \htri [b]$ by the second part. This is a contradiction by statement (1).\qedhere
		\end{itemize}
	\end{proof}
	
	\begin{lemma}\label{thm:oc-atoms-comp}
		If $G \hleq [a]$ for some option closed game $G$ and atom $a$, then $G$ is atomic.
	\end{lemma}
	
	\begin{proof}
		Suppose $G$ is not atomic. Then, since $G$ is option closed, there exists an atomic left option $[b]$ of $G$. It follows that we have $[b] \htri G \hleq [a]$, which is a contradiction by Lemma~\ref{thm:atoms-not-comp}. So $G$ is atomic.
	\end{proof}
	
	\begin{lemma}\label{thm:oc-right-closure}
		Let $G$ be an option closed game, then for any game $H$, if $G^R \htri H$ for some right option $G^R$ of $G$, then $G \htri H$.
	\end{lemma}
	
	\begin{proof}
		Suppose $G^R \htri H$. There are three ways in which this can happen.
		\begin{itemize}
			\item [] Case i) $G^{RR} \hleq H$. But, $G^{RR}$ is a right option of $G$ since $G$ is option closed. So we get $G \htri H$.
			\item [] Case ii) $G^R \hleq H^L$. It follows that $G \htri H$.
			\item [] Case iii) $G^{RR} \hleq H^L$. But, $G^{RR}$ is a right option of $G$ since $G$ is option closed. It follows that $G \htri H$.\qedhere
		\end{itemize}
	\end{proof}
	
	\begin{lemma}\label{thm:oc-left-closure}
		Let $H$ be an option closed game, then for any game $G$ if $G \htri H^L$ for some left option $H^L$ of $H$, then $G \htri H$.
	\end{lemma}
	
	\begin{proof}
		The proof is similar to that of Lemma~\ref{thm:oc-right-closure}.
	\end{proof}
	
	\begin{theorem}\label{thm:oc-transitivity}
		For option closed $G$, $H$, and $K$
		\begin{itemize}
			\item[1)] $G\hleq H \hleq K \Rightarrow G \hleq K$
			\item[2)] $G\htri H \hleq K \Rightarrow G \htri K$
			\item[3)] $G\hleq H \htri K \Rightarrow G \htri K$
		\end{itemize}
	\end{theorem}
	
	\begin{proof}
		We prove these properties by mutual induction.
		\begin{itemize}
			\item [1)] First, consider any $G^L$. Then $G^L \htri H \hleq K$. By induction hypothesis (2), $G^L \htri K$. Next, consider any $K^R$. Then $G \hleq H \htri K^R$. By induction hypothesis (3), $G \htri K^R$. Finally, suppose $G= [a]$ and $K = [c]$ are atomic. Then by Lemma~\ref{thm:oc-atoms-comp}, $H = [b]$ is also atomic. So we have that $a \leq b \leq c$ and since $\leq$ is a transitive relation over the poset of outcomes, we have that $a\leq c$. Thus $G \hleq K$.
			\item [2)] There are three possible ways we can have $G \htri H$.
			\begin{itemize}
				\item [] Case i) $G^R \hleq H$. Then we have $G^R \hleq H \hleq K$. By induction hypothesis (1), $G^R \hleq K$. Hence, $G \htri K$.
				\item [] Case ii) $G \hleq H^L$. Then we have $G \hleq H^L \htri K$. By induction hypothesis (1), $G \htri K$.
				\item [] Case iii) $G^R \hleq H^L$. Then we have $G^R \hleq H^L \htri K$. By induction hypothesis (1), $G^R \htri K$. By Lemma~\ref{thm:oc-right-closure} it follows that $G \htri K$.
			\end{itemize}
			\item [3)] There are three possible ways we can have $H \htri K$.
			\begin{itemize}
				\item [] Case i) $H^R \hleq K$. Then we have $G \htri H^R \htri K$. By induction hypothesis (1), $G \htri K$.
				\item [] Case ii) $H \hleq K^L$. Then we have $G \hleq H \hleq K^L$. By induction hypothesis (1), $G \hleq K^L$. Hence, $G \htri K$.
				\item [] Case iii) $H^R \hleq K^L$. Then we have $G \htri H^R \htri K^L$. By induction hypothesis (1), $G \htri K^L$. By Lemma~\ref{thm:oc-left-closure} it follows that $G \htri K$.\qedhere
			\end{itemize}
		\end{itemize}
	\end{proof}
	
\chapter{Weaker Versions of Option Closure}\label{chap:weaker-option-closure}

	\section{The Canonical Form Problem}
	
	When working with Rex+ positions we would like to be able to work with unique canonical forms as we do in normal play. However, as soon as we attempt to apply similar reductions we lose the property of option closure, as we are removing and replacing options.
	
	Currently, our order is only transitive over option closed games, but when talking about canonical forms of option closed games we lose the property of option closure, and thus our ability to use transitivity. However, without transitivity, the idea of a canonical form uses its utility, as canonical forms are only equivalent to their original games, and without transitivity the notion of equivalence is largely meaningless.
	
	So we need to find some property of games that option closed games have, that is preserved by our reductions, and allows for transitivity. In this chapter we present several attempts, and the class of games we eventually found that satisfies these conditions.
	
	\section{Weak Option Closure}
	
	The first property we tried was a property we call weak option closure. 
	
	\begin{definition}
		A game $G$ is called \emph{locally weakly option closed} if for every left option $G^L$ of $G$, $X \htri  G^L \Rightarrow X \htri G$ and for every right option $G^R$ of $G$, $G^R \htri X \Rightarrow G \htri X$
	\end{definition}
	
	\begin{definition}\label{def:weak-option-closure}
		A game $G$ is called \emph{weakly option closed} if it is locally weakly option closed and all of its options are weakly option closed. 
	\end{definition}
	
	These properties are used in proving transitivity, so by weakening our notion of option closure to only these properties we do not lose transitivity as we will proceed to show. 
	
	\begin{lemma} \label{thm:woc-multi-left-option-triangle}
		For weakly option closed $G$, given a game $X$ such that $X \htri G^{L...L}$, then $X \htri G$.
	\end{lemma}
	
	\begin{proof}
		We consider the two possibilities for what $G^{L...L}$ is.
		\begin{itemize}
			\item [] Case i) We have $X \htri G^L$. Then by Definition~\ref{def:weak-option-closure} it follows that $X \htri G$.
			\item [] Case ii) We have $X \htri G^{L...LL}$. Then, since $G^{L...L}$ is weakly option closed, it follows from Definition~\ref{def:weak-option-closure} that $X \htri G^{L...L}$. Then, by induction hypothesis, $X \htri G$.\qedhere
		\end{itemize}
	\end{proof}
	
	\begin{lemma}\label{thm:woc-multi-right-option-triangle}
		For weakly option closed $G$, given a game $X$ such that $G^{R...R} \htri X$, then $G \htri X$.
	\end{lemma}
	
	\begin{proof}
		The proof follows similarly to the proof of Lemma~\ref{thm:woc-multi-left-option-triangle}.
	\end{proof}
	
	\begin{lemma}\label{thm:woc-atom-triangle-existence}
		For weakly option closed $G$, there exist atoms $[a]$ and $[b]$ such that $[a] \htri G$ and $G \htri [b]$.
	\end{lemma}
	
	\begin{proof}
		Since $G$ is a well-defined game it has some $G^{L...LL} = [a]$ which is atomic. Then $[a]$ is a left option of $G^{L...L}$. It follows that $[a]\htri G^{L...L}$. Then, by Lemma~\ref{thm:woc-multi-left-option-triangle}, it follows that $[a] \htri G$.
		
		A similar proof shows that there is an atomic game $[b]$ such that $G \htri [b]$.
	\end{proof}
	
	\begin{lemma}\label{thm:woc-atoms-comp}
		If $G \hleq [a]$ for some weakly option closed game $G$ and atom $a$, then $G$ is atomic.
	\end{lemma}
	
	\begin{proof}
		Suppose $G$ is not atomic. Then by Lemma~\ref{thm:woc-atom-triangle-existence} there exists an atomic game $[b]$ with $[b] \htri G$. Then we have $[b] \htri G \hleq [a]$, which is a contradiction by Lemma~\ref{thm:atoms-not-comp}. So $G$ is atomic.
	\end{proof}
	
	\begin{theorem}\label{thm:woc-transitivity}
		For weakly option closed $G$, $H$, and $K$
		\begin{itemize}
			\item[1)] $G\hleq H \hleq K \Rightarrow G \hleq K$.
			\item[2)] $G\htri H \hleq K \Rightarrow G \htri K$.
			\item[3)] $G\hleq H \htri K \Rightarrow G \htri K$.
		\end{itemize}
	\end{theorem}
	
	\begin{proof}
		The proof is identical to that of Theorem~\ref{thm:oc-transitivity}, using Definition~\ref{def:weak-option-closure} and Lemma~\ref{thm:woc-atoms-comp} instead of Lemmas~\ref{thm:oc-atoms-comp},~\ref{thm:oc-right-closure}, and~\ref{thm:oc-left-closure}.	
	\end{proof}
	
	However when attempting to show that weakly option closed games were closed under our game form reductions we ran into problems.
	
	After weak option closure, we trialled many possible properties. The problem with weak option closure is that we have in some sense that weak option closed games are ``pseudo closed" to every game, i.e., they have this property with respect to every game. However, we really needed a class of games which were ``pseudo closed" with respect to themselves. But trying to define such a class presents a chicken and egg problem. We eventually did reach a solution, which will be covered in the next section.
	
	\section{The Classes $\C$ and $\D$}
	
	We begin with some definitions, to formalize our notion of ``pseudo closure".
	
	\begin{definition}
		Let $\cal{A}$ be a class of games. We say a game $G$ is $\cal{A}$-closed if $\forall X \in \cal{A}$ and $\forall G^L, G^R$ we have:
		\begin{itemize}
			\item [1)]$G^R \htri X \Rightarrow G \htri X$ and;
			\item [2)] $X \htri G^L \Rightarrow X \htri G$.
		\end{itemize}
	\end{definition}
	
	\begin{definition}
		We call the class of games which are $\cal{A}$-closed, $\cal{A}^*$.
	\end{definition}
	
	Then we define a function that we call $*$. This is simply the function that sends a class of games $\cal{A}$ to the class of games which is $\cal{A}$-closed. 
	
	\begin{definition}
		We define the function $*$ from the power class of all games to itself, as follows. Given a class of games $\A$: $*(\A) = \A^*$. We write $*(\A)$ as $\A^*$.
	\end{definition}
	
	We intend to use the Knaster-Tarski fixed point theorem, see Theorem~\ref{thm:knaster-tarski}, to find our class of games which is closed to itself.
	
	However, there is a small problem, the function $*$ is not itself a monotone function, a requirement for the Knaster-Tarski fixed point theorem. Thankfully, the function $**$ is monotone, as we show in the following lemma.
	
	\begin{lemma}
		The function $**$ is monotone.
	\end{lemma}
	
	\begin{proof}
		Suppose we have classes of games $\A$ and $\B$ with $\A \subseteq \B$. Then, every game in $\B^*$ is closed with respect to every game in $\A$, since $\A$ is a subset of $\B$. Thus every game in $\B^*$ is also in $\A^*$. Thus $\B^* \subseteq \A^*$.
		
		But then, by a similar argument, $\A^{**} \subseteq \B^{**}$. So $**$ is a monotone function.
	\end{proof}
	
	Then, by Theorem~\ref{thm:knaster-tarski} we know that $**$ has a least fixed point. 
	
	\begin{definition}\label{def:classes-c-and-d}
		Let $\cal{C}$ be the least fixed point of $^{**}$ and let $\cal{D} = \cal{C}^*$. A game $G$ is in $\cal{C}$ if $\forall X \in \cal{D}$ and $\forall G^L, G^R$
		\begin{itemize}
			\item [1)] $G^R \htri X \Rightarrow G \htri X$ and;
			\item [2)] $X \htri G^L \Rightarrow X \htri G$.
		\end{itemize}
		A game $G$ is in $\cal{D}$ if $\forall X \in \cal{C}$ and $\forall G^L, G^R$
		\begin{itemize}
			\item [1)] $G^R \htri X \Rightarrow G \htri X$ and;
			\item [2)] $X \htri G^L \Rightarrow X \htri G$.
		\end{itemize}
	\end{definition}
	
	\begin{lemma}
		The class $\C$ is contained in the class $\D$.
	\end{lemma}
	
	\begin{proof}
		We show first that $\D$ is a fixed point of $**$. Consider $\D^{**}$. It is equal to $\C^{***}$. But since $\C$ is a fixed point of $**$, we have that $\C{***} = \C^*$. So $D^{**} = \C^{***} = \C^* = \D$. Thus $\D$ is a fixed point of $**$. Thus, since $\C$ is the least fixed point of $**$, we have that $\C \subseteq \D$.
	\end{proof}
	
	First, note that all option closed games are in $\C$. We will show in the following section and chapter that the class $\C$ has both other properties we wish it to have.
	
	\section{Transitivity}
	
	First we show that games in $\C$ have transitivity.
	
	\begin{lemma}\label{thm:c-left-triangle-closure}
		For a game $G$ in $\C$, given a game $X$ in $\D$ such that $X \htri G^{L...L}$, then $X \htri G$.
	\end{lemma}
	
	\begin{proof}
		We consider two possibilities for what $G^{L...L}$ could be.
		\begin{itemize}
			\item [] Case i) We have $X \htri G^L$. Then by Definition~\ref{def:classes-c-and-d} it follows that $X \htri G$.
			\item[] Case ii) We have $X \htri G^{L...LL}$. Then, since $G^{L...LL}$ is in $\C = \D^*$, it follows that $X \htri G^{L...L}$. Then, by induction hypothesis $X \htri G$.\qedhere
		\end{itemize}
	\end{proof}
	
	\begin{lemma}\label{thm:c-right-triangle-closure}
		For a game $G$ in $\C$, given a game $X$ in $\D$ such that $G^{R...R} \htri X$, then $G \htri X$.
	\end{lemma}
	
	\begin{proof}
		The proof follows similarly to the proof of Lemma~\ref{thm:c-left-triangle-closure}.
	\end{proof}
	
	\begin{lemma}\label{thm:c-atoms-triangle-existence}
		Given a game $G$ that is in $\D$ there exists atomic games $[a]$ and $[b]$ such that $[a] \htri G$ and $G \htri [b]$.
	\end{lemma}
	
	\begin{proof}
		Since $G$ is a well-defined game it has some $G^{L...LL} = [a]$ which is atomic. Then $[a]$ is a left option of $G^{L...L}$. It follows that $[a]\htri G^{L...L}$. Then, by Lemma~\ref{thm:c-left-triangle-closure}, it follows that $[a] \htri G$.
		
		A similar proof shows that there is an atomic game $[b]$ such that $G \htri [b]$.
	\end{proof}
	
	\begin{lemma}\label{thm:c-atom-comp}
		If $G \hleq [a]$ for some game $G$ in $\C$ and atom $a$, then $G$ is atomic.
	\end{lemma}
	
	\begin{proof}
		Suppose $G$ is not atomic. Then by Lemma~\ref{thm:c-atoms-triangle-existence} there exists an atomic game $[b]$ with $[b] \htri G$. Then we have $[b] \htri G \hleq [a]$, which is a contradiction by Lemma~\ref{thm:atoms-not-comp}. So $G$ is atomic.
	\end{proof}
	
	\begin{theorem}\label{thm:c-transitivity-D-right}
		For games $G$ in $\C$ and $H$, $K$ in $\D$ we have:
		\begin{itemize}
			\item[1)] $G\hleq H \hleq K \Rightarrow G \hleq K$.
			\item[2)] $G\htri H \hleq K \Rightarrow G \htri K$.
			\item[3)] $G\hleq H \htri K \Rightarrow G \htri K$.
		\end{itemize}
	\end{theorem}
	
	\begin{proof}
		We prove by mutual induction.
		
		\begin{itemize}
			\item[1)] First, consider any $G^L$. Then $G^L \htri H \hleq K$. By induction hypothesis (2), $G^L \htri K$. Next, consider any $K^R$. Then $G \hleq H \htri K^R$. By induction hypothesis (3), $G \htri K^R$. Finally, suppose $G = [a]$ and $K = [c]$ are atomic. Then by Lemma~\ref{thm:c-atom-comp}, $H = [b]$ is also atomic. So we have that $a \leq b \leq c$ and since $\leq$ is a transitive relation over the poset of outcomes, we have that $a\leq c$. Thus $G \hleq K$.
			
			\item[2)] There are three possible ways we can have $G \htri H$.
			
			\begin{itemize}
				\item [] Case i) $G^R \hleq H$. Then we have $G^R \hleq H \hleq K$. By induction hypothesis (1), $G^R \hleq K$. Hence, $G \htri K$.
				
				\item [] Case ii) $G \hleq H^L$. Then we have $G \hleq H^L \htri K$. By induction hypothesis (1), $G \htri K$.
				
				\item [] Case iii) $G^R \hleq H^L$. Then we have $G^R \hleq H^L \htri K$. By induction hypothesis (1), $G^R \htri K$. By Definition~\ref{def:classes-c-and-d} it follows that $G \htri K$.
			\end{itemize}		
			\item[3)] There are three possible ways we can have $H \htri K$.
			
			\begin{itemize}
				\item [] Case i) $H^R \hleq K$. Then we have $G \htri H^R \htri K$. By induction hypothesis (1), $G \htri K$.
				
				\item [] Case ii) $H \hleq K^L$. Then we have $G \hleq H \hleq K^L$. By induction hypothesis (1), $G \hleq K^L$. Hence, $G \htri K$.
				
				\item [] Case iii) $H^R \hleq K^L$. Then we have $G \htri H^R \htri K^L$. By induction hypothesis (1), $G \htri K^L$. By Definition~\ref{def:classes-c-and-d} it follows that $G \htri K$.\qedhere
			\end{itemize}
		\end{itemize}
	\end{proof}
	
	Similarly we have the corresponding theorem where we switch which of $G$ and $K$ are in $\C$ and $\D$.
	
	\begin{theorem}\label{thm:c-transitivity-D-left}
		For games $G$, $H$ in $\D$ and $K$ in $\C$ we have:
		\begin{itemize}
			\item[1)] $G\hleq H \hleq K \Rightarrow G \hleq K$.
			\item[2)] $G\htri H \hleq K \Rightarrow G \htri K$.
			\item[3)] $G\hleq H \htri K \Rightarrow G \htri K$.
		\end{itemize}
	\end{theorem}
	
	\begin{proof}
		The proof follows similarly for that of Theorem~\ref{thm:c-transitivity-D-right}.	
	\end{proof}
	
\chapter{Canonical Forms}\label{chap:canonical-forms}
	\section{Reductions}
	
	In this chapter we will introduce seven reductions that we use to simplify game forms for Rex+ positions. The first two reductions, domination and reversibility are analogous to those we saw earlier for normal play, but the other five are entirely new, and are needed because of the option closed nature of our order and games.
	
	The most interesting sections for those already familiar with game form reductions are sections~\ref{subsec:reversible-domination} and \ref{subsec:self-reversible-domination}, where we cover reversible domination and self-reversible domination. The proofs covered in sections \ref{subsec:two-reversibility}, \ref{subsec:two-reversible-domination}, and \ref{subsec:two-self-reversible-domination} are very similar to those covered in \ref{subsec:reversibility}, \ref{subsec:reversible-domination}, and \ref{subsec:self-reversible-domination}.
		
	\subsection{Domination}\label{subsec:domination}
	
	\begin{lemma}\label{thm:domination-X-G}
		Let $G \in \cal{C}$, with $G = \{A, B, G^{\cal{L}}\mid G^{\cal{R}}\} $, and $A \hleq B$. Let \\$G'~=~\{B, G^{\cal{L}}\mid G^{\cal{R}}\}$. Then, for any game $X \in \D$:
		\begin{itemize}
			\item [1)] $X \hleq G \Rightarrow X \hleq G'$.
			\item [2)] $X \htri G \Rightarrow X \htri G'$.
		\end{itemize}
	\end{lemma}
	
	\begin{proof}
		We prove these statements by mutual induction. Take $X \in \cal{D}$.
		\begin{itemize}
			\item [1)] Take $X^L$. $X^L \htri G \Rightarrow X^L \htri G'$ by induction hypothesis (2). Take $G'^R$. $G'^R$ is a right option of $G$, thus because $X \hleq G$, we have $X \htri G'^R$. $G'$ is non atomic, so we have $X \hleq G'$.
			\item [2)] There are three ways we can have $X \htri G$.
			\begin{itemize}
				\item [] Case i) $X^R \hleq G$. Then by induction hypothesis (1) we have $X^R \hleq G'$. Thus $X \htri G'$.
				\item [] Case ii) $X \hleq G^L$. Then we have two options for what $G^L$ could be.
				\begin{itemize}
					\item [] Case a) $G^L = A$. Then we have $X \hleq A \hleq B$. By Theorem~\ref{thm:c-transitivity-D-left} we have $X \hleq B$. Since $B$ is a left option of $G'$ we have $X \htri G'$.
					\item [] Case b) $G^L \neq A$. Then $G^L$ is a left option of $G'$. Thus $X \htri G'$.
				\end{itemize}
				\item[] Case iii) $X^R \hleq G^L$. We again have two options for what $G^L$ could be.
				\begin{itemize}
					\item [] Case a) $G^L = A$. Then we have $X^R \hleq A \hleq B$. By Theorem~\ref{thm:c-transitivity-D-left} we have  $X^R \hleq B$. Since $B$ is a left option of $G"'$ we have $X \htri G'$.
					\item [] Case b) $G^L \neq A$. Then $G^L$ is a left option of $G'$. Thus $X \htri G'$.\qedhere
				\end{itemize}
			\end{itemize}
		\end{itemize}
	\end{proof}
	
	\begin{lemma}\label{thm:domination-G-X}
		Let $G \in \cal{C}$, with $G = \{A, B, G^{\cal{L}}\mid G^{\cal{R}}\} $, and $A \hleq B$. Let \\$G' = \{B, G^{\cal{L}}\mid G^{\cal{R}}\}$. Then, for any game $X \in \cal{D}$:
		\begin{itemize}
			\item [1)] $G \hleq X \Rightarrow G' \hleq X$.
			\item [1)] $G \htri X \Rightarrow G' \htri X$.
		\end{itemize}
	\end{lemma}
	
	\begin{proof}
		We prove these statements by mutual induction. Take $X \in \cal{D}$.
		\begin{itemize}
			\item [1)] Take $G'^L$. $G'^L$ is a left option of $G$, thus because $G \hleq X$, we have $G'^L \htri X$. Take $X^R$. $G \htri X^R \Rightarrow G' \htri X^R$ by induction hypothesis (2) . $G'$ is non atomic, so we have $G' \hleq X$.
			\item [2)] There are three ways we can have $G \htri X$.
			\begin{itemize}
				\item [] Case i) $G^R \hleq X$. $G^R$ is a right option of $G'$, so we have $G' \htri X$.
				\item [] Case ii) $G \hleq X^L$. Then by induction hypothesis (1) we have $G' \hleq X^L$. Thus $G' \htri X$.
				\item [] Case iii) $G^R \hleq X^L$. $G^R$ is a right option of $G'$, so we have $G' \htri X$.\qedhere
			\end{itemize}
		\end{itemize}
	\end{proof}
	
	\begin{theorem}\label{thm:domination-preserves-c}
		Let $G \in \C$, with $G = \{A, B, G^{\Le}\mid G^{\R}\} $, and $A \hleq B$. Let \\$G' = \{B, G^{\cal{L}}\mid G^{\cal{R}}\}$. Then $G' \in \cal{C}$.
	\end{theorem}
	
	\begin{proof}
		Let $X \in \cal{D}$ with $G'^R \htri X$ for some $G'^R$. We wish to show $G' \htri X$.
		
		First, we note that $G'^R$ is a right option of $G$. Since $G \in \cal{C}$, by Definition~\ref{def:classes-c-and-d} it follows that $G \htri X$. By Lemma~\ref{thm:domination-G-X}, we have that $G' \htri X$.
		
		Let $X \in \cal{D}$ with $X \htri G'^L$ for some $G'^L$. We wish to show $X \htri G'$.
		
		We note that $G'^L$ is a left option of $G$. Since $G \in \cal{C}$, by Definition~\ref{def:classes-c-and-d} it follows that $X \htri G$. By Lemma~\ref{thm:domination-X-G}, we have that $X \htri G'$.
		
		Thus, $G' \in \cal{C}$.
	\end{proof}
	
	\begin{theorem}\label{thm:domination-equal}
		Let $G \in \cal{C}$, with $G = \{A, B, G^{\cal{L}}\mid G^{\cal{R}}\} $, and $A \hleq B$. Then \\$G \heq \{B, G^{\cal{L}}\mid G^{\cal{R}}\}$.
	\end{theorem}
	
	\begin{proof}
		First we show $G \hleq G'$. Take $G^L$. There are two options for what $G^L$ could be.
		\begin{itemize}
			\item [] Case 1) $G^L = A$. Then $G^L = A \hleq B$ where $B$ is a left option of $G'$. So $G^L \htri G'$
			\item [] Case 2) $G^L \neq A$. Then $G^L$ is a left option of $G'$ so $G^L \htri G'$.
		\end{itemize}
		Take $G'^R$. $G'^R$ is a right option of $G$, so $G\htri G'^R$. Neither $G$ nor $G'$ is atomic, so we have $G \hleq G'$.
		
		Next we show $G' \hleq G$. Take $G'^L$. $G'^L$ is a left option of $G$, so $G'^L\htri G$. Take $G^R$. $G^R$ is a right option of $G'$, so $G'\htri G^R$. Neither $G$ nor $G'$ is atomic, so we have $G' \hleq G$. Thus $G \heq G'$.
	\end{proof}
	
	\subsection{Reversibility}\label{subsec:reversibility}
	
	\begin{lemma}\label{thm:reversibility-X-G}
		Let $G \in \cal{C}$, with $G = \{A,~G^{\cal{L}}\mid~G^{\cal{R}}\} $, and $A^R \hleq G$. Moreover, let $G'~=~\{A^{R\Le},~G^{\cal{L}}\mid~G^{\cal{R}}\}$. Then, for any game $X \in \cal{D}$:
		\begin{itemize}
			\item [1)] $X \hleq G \Rightarrow X \hleq G'$.
			\item [1)] $X \htri G \Rightarrow X \htri G'$.
		\end{itemize}
	\end{lemma}
	
	\begin{proof}
		We prove these statements by mutual induction. Take $X \in \cal{D}$.
		\begin{itemize}
			\item [1)] Take $X^L$. $X^L \htri G \Rightarrow X^L \htri G'$ by induction hypothesis (2). Take $G'^R$. $G'^R$ is a right option of $G$, thus because $X \hleq G$, we have $X \htri G'^R$. $G'$ is non atomic, so we have $X \hleq G'$.
			\item [2)] There are three ways we can have $X \htri G$.
			\begin{itemize}
				\item [] Case i) $X^R \hleq G$. Then by induction hypothesis (1) we have $X^R \hleq G'$. Thus $X \htri G'$.
				\item [] Case ii) $X \hleq G^L$. Then there are two options for what $G^L$ could be.
				\begin{itemize}
					\item [] Case a) $G^L = A$. Then we have $X \htri A^R$. There are three ways this is possible.
					\begin{itemize}
						\item [] Case I) $X^R \hleq A^R$. But $A^R \hleq G$. By Theorem~\ref{thm:c-transitivity-D-left} we have that $X^R \hleq G$. By induction hypothesis (1), we then have $X^R \hleq G'$. Thus $X \htri G'$.
						\item [] Case II) $X \hleq A^{RL}$. $A^{RL}$ is a left option of $G'$ so it follows that $X \htri G'$.
						\item [] Case III) $X^R \hleq A^{RL}$. $A^{RL}$ is a left option of $G'$ so it follows that $X \htri G'$.
					\end{itemize}
					\item [] Case b) $G^L \neq A$. Then $G^L$ is a left option of $G'$. Thus $X \htri G'$.
				\end{itemize}
				\item [] Case iii) $X^R \hleq G^L$. There are again two options for what $G^L$ could be.
				\begin{itemize}
					\item [] Case a) $G^L = A$. Then we have $X^R \htri A^R$. Since $X^R \in \cal{D}$ and $A^R \in \cal{C}$, by Definition~\ref{def:classes-c-and-d} it follows that $X \htri A^R$. There are three ways this is possible
					\begin{itemize}
						\item [] Case I) $X^R \hleq A^R$. But $A^R \hleq G$. By Theorem~\ref{thm:c-transitivity-D-left} we have that $X^R \hleq G$. By induction hypothesis (1), we then have $X^R \hleq G'$. Thus $X \htri G'$.
						\item [] Case II) $X \hleq A^{RL}$. $A^{RL}$ is a left option of $G'$ so it follows that $X \htri G'$.
						\item [] Case III) $X^R \hleq A^{RL}$. $A^{RL}$ is a left option of $G'$ so it follows that $X \htri G'$.
					\end{itemize}
					\item [] Case b) $G^L \neq A$. Then $G^L$ is a left option of $G'$. Thus $X \htri G'$.\qedhere
				\end{itemize}
			\end{itemize}
		\end{itemize}
	\end{proof}
	
	\begin{lemma}\label{thm:reversibility-G-X}
		Let $G \in \cal{C}$, with $G = \{A, G^{\cal{L}}\mid G^{\cal{R}}\} $, and $A^R \hleq G$. Let \\$G' = \{A^{R\Le}, G^{\cal{L}}\mid G^{\cal{R}}\} $. Then, for any game $X \in \cal{D}$:
		\begin{itemize}
			\item [1)] $G \hleq X \Rightarrow G' \hleq X$.
			\item [1)] $G \htri X \Rightarrow G' \htri X$.
		\end{itemize}
	\end{lemma}
	
	\begin{proof}
		We prove these statements by mutual induction. Take $X \in \cal{D}$.
		\begin{itemize}
			\item [1)] Take $G'^L$. There are two options for what $G'^L$ could be.
			\begin{itemize}
				\item [] Case i) $G'^L = A^{RL}$. Since $A^R \hleq G$ it follows that $A^{RL} \htri G \hleq X$. By Theorem~\ref{thm:c-transitivity-D-right}, so $A^{RL} \htri X$.
				\item [] Case ii) $G'^L \neq A^{RL}$. Then $G'^L$ is a left option of $G$. Because $G \hleq X$, it follows that $G'^L \htri X$.
			\end{itemize}
			Now take $X^R$. $G \htri X^R \Rightarrow G' \htri X^R$ by induction hypothesis (2). $G'$ is non atomic, so we have $G' \hleq X$.
			\item [2)] There are three ways we can have $G \htri X$.
			\begin{itemize}
				\item [] Case i) $G^R \hleq X$. $G^R$ is a right option of $G'$, so we have $G' \htri X$.
				\item [] Case ii) $G \hleq X^L$. Then by induction hypothesis (1) we have $G' \hleq X^L$. Thus $G' \htri X$.
				\item [] Case iii) $G^R \hleq X^L$. $G^R$ is a right option of $G'$, so we have $G' \htri X$.\qedhere
			\end{itemize}
		\end{itemize}
	\end{proof}
	
	\begin{theorem}\label{thm:reversiblity-preserves-c}
		Let $G \in \cal{C}$, with $G = \{A, G^{\cal{L}}\mid G^{\cal{R}}\} $, and $A^R \hleq G$. Let \\$G' = \{A^{R\Le}, G^{\cal{L}}\mid G^{\cal{R}}\}$. Then $G' \in \cal{C}$.
	\end{theorem}
	
	\begin{proof}
		Let $X \in \cal{D}$ with $G'^R \htri X$ for some $G'^R$. We wish to show $G' \htri X$. First, we note that $G'^R$ is a right option of $G$. By Theorem~\ref{thm:c-transitivity-D-right} we then have that $G \htri X$. By Lemma~\ref{thm:reversibility-G-X}, it follows that $G' \htri X$.
		
		Let $X \in \cal{D}$ with $X \htri G'^L$ for some $G'^L$. We wish to show $X \htri G'$. There are two options for what $G'^L$ is.
		\begin{itemize}
			\item [] Case i) $G'^L = A^{RL}$. By Definition~\ref{def:classes-c-and-d} $X \htri A^R$. Since $A^R \hleq G$, we have that $X \htri G$ by Theorem~\ref{thm:c-transitivity-D-left}. It follows from Lemma~\ref{thm:reversibility-X-G} that $X \htri G'$. 
			\item [] Case ii) $G'^L \neq A^{RL}$. Then $G'^L$ is a left option of $G$. Since $G \in \cal{C}$, by Definition~\ref{def:classes-c-and-d} we have that $X \htri G$. By Lemma~\ref{thm:reversibility-X-G}, it follows that $X \htri G'$.
		\end{itemize}
		Thus, $G' \in \cal{C}$.
	\end{proof}
	
	\begin{theorem}\label{thm:reversibility-equal}
		Let $G \in \cal{C}$, with $G = \{A, G^{\cal{L}}\mid G^{\cal{R}}\} $, and $A^R \hleq G$. Then \\$G \heq \{A^{R\Le}, G^{\cal{L}}\mid G^{\cal{R}}\} $.
	\end{theorem}
	
	\begin{proof}
		First we show $G \hleq G'$. Take $G^L$. There are two options for what $G^L$ can be.
		\begin{itemize}
			\item [] Case i) $G^L = A$. Then since $A^R \hleq G$ we have that $A^R \hleq G'$ by Lemma~\ref{thm:reversibility-X-G}. Thus $A \htri G'$ by Definition~\ref{def:classes-c-and-d}.
			\item [] Case ii) $G^L \neq A$. Then $G^L$ is a left option of $G'$ so $G^L \htri G'$.
		\end{itemize}
		Next, take $G'^R$. $G'^R$ is a right option of $G$, so $G\htri G'^R$. Neither $G$ nor $G'$ is atomic, so we have $G \hleq G'$.
		
		Now we show $G' \hleq G$. Take $G'^L$. There are also two options for what $G'^L$ can be.
		\begin{itemize}
			\item [] Case i) $G'^L = A^{RL}$. Then since $A^R \hleq G$ it follows that $A^{RL} \htri G$.
			\item [] Case ii) $G'^L \neq A^{RL}$. Then $G'^L$ is a left option of $G$, so $G'^L\htri G$.
		\end{itemize}
		Next, take $G^R$. $G^R$ is a right option of $G'$, so $G'\htri G^R$. Neither $G$ nor $G'$ is atomic, so we have $G' \hleq G$. 
		
		Thus $G \heq G'$.
	\end{proof}
	
	\subsection{Reversible Domination}\label{subsec:reversible-domination}
	
	\begin{lemma}\label{thm:rev-dom-X-G}
		Let $G \in \cal{C}$, with $G = \{A, B, G^{\cal{L}}\mid G^{\cal{R}}\} $, and $A^R \hleq B$. Let \\$G'~=~\{A^{R\Le}, B, G^{\cal{L}}\mid G^{\cal{R}}\}$. Then, for any game $X \in \cal{D}$:
		\begin{itemize}
			\item [1)] $X \hleq G \Rightarrow X \hleq G'$.
			\item [1)] $X \htri G \Rightarrow X \htri G'$.
		\end{itemize}
	\end{lemma}
	
	\begin{proof}
		We prove these statements by mutual induction. Take $X \in \cal{D}$.
		\begin{itemize}
			\item [1)] Take $X^L$. $X^L \htri G \Rightarrow X^L \htri G'$ by induction hypothesis (2). Take $G'^R$. $G'^R$ is a right option of $G$, so we have $X \htri G'^R$. $G'$ is non atomic, so we have $X \hleq G'$.
			\item [2)] There are three ways we can have $X \htri G$.
			\begin{itemize}
				\item [] Case i) $X^R \hleq G$. Then by induction hypothesis (1) we have $X^R \hleq G'$. Thus $X \htri G'$.
				\item [] Case ii) $X \hleq G^L$. Then there are two options for $G^L$.
				\begin{itemize}
					\item [] Case a) $G^L = A$. Then we have $X \htri A^R$. There are three ways this is possible
					\begin{itemize}
						\item [] Case I) $X^R \hleq A^R$. But $A^R \hleq B$. By Theorem~\ref{thm:c-transitivity-D-left} we have that $X^R \hleq B$. Since $B$ is a left option of $G'$ it follows that $X \htri G'$.
						\item [] Case II) $X \hleq A^{RL}$. $A^{RL}$ is a left option of $G'$ so it follows that $X \htri G'$.
						\item [] Case III) $X^R \hleq A^{RL}$. $A^{RL}$ is a left option of $G'$ so it follows that $X \htri G'$.
					\end{itemize}
					\item [] Case b) $G^L \neq A$. Then $G^L$ is a left option of $G'$. Thus $X \htri G'$.
				\end{itemize}
				\item [] Case iii) $X^R \hleq G^L$. There are again two options for $G^L$.
				\begin{itemize}
					\item [] Case a) $G^L = A$. Then we have $X^R \htri A^R$. By Definition~\ref{def:classes-c-and-d} it follows that $X \htri A^R$. There are three ways this is possible.
					\begin{itemize}
						\item [] Case I) $X^R \hleq A^R$. But $A^R \hleq B$. By Theorem~\ref{thm:c-transitivity-D-left} we have that $X^R \hleq B$. Since $B$ is a left option of $G'$ it follows that $X \htri G'$.
						\item [] Case II) $X \hleq A^{RL}$. $A^{RL}$ is a left option of $G'$ so it follows that $X \htri G'$.
						\item [] Case III) $X^R \hleq A^{RL}$. $A^{RL}$ is a left option of $G'$ so it follows that $X \htri G'$.
					\end{itemize}
					\item [] Case b) $G^L \neq A$. Then $G^L$ is a left option of $G'$. Thus $X \htri G'$.\qedhere
				\end{itemize}
			\end{itemize}
		\end{itemize}
	\end{proof}
	
	\begin{lemma}\label{thm:rev-dom-G-X}
		Let $G \in \cal{C}$, with $G = \{A, B, G^{\cal{L}}\mid G^{\cal{R}}\} $, and $A^R \hleq B$. Let \\$G' = \{A^{R\Le}, B, G^{\cal{L}}\mid G^{\cal{R}}\}$. Then, for any game $X \in \cal{D}$:
		\begin{itemize}
			\item [1)] $G \hleq X \Rightarrow G' \hleq X$.
			\item [1)] $G \htri X \Rightarrow G' \htri X$.
		\end{itemize}
	\end{lemma}
	
	\begin{proof}
		We prove these statements by mutual induction. Take $X \in \cal{D}$.
		\begin{itemize}
			\item [1)] Take $G'^L$. There are two possible options for $G'^L$.
			\begin{itemize}
				\item [] Case i) $G'^L = A^{RL}$. Since $A^R \hleq B$ we have that $A^{RL} \htri B$. Since $B$ is a left option of $G$ it follows from Definition~\ref{def:classes-c-and-d} that $A^{RL} \htri G$. But $G \hleq X$. By Theorem~\ref{thm:c-transitivity-D-right} we have that $A^{RL} \htri X$.
				\item [] Case ii) $G'^L \neq A^{RL}$. Then $G'^L$ is a left option of $G$. Because $G \hleq X$, it follows that $G'^L \htri X$.
			\end{itemize}
			Next, take $X^R$. $G \htri X^R \Rightarrow G' \htri X^R$ by induction hypothesis (2). $G'$ is non atomic, so we have $G' \hleq X$.
			\item [2)] There are three ways we can have $G \htri X$.
			\begin{itemize}
				\item [] Case i) $G^R \hleq X$. $G^R$ is a right option of $G'$, so we have $G' \htri X$.
				\item [] Case ii) $G \hleq X^L$. Then by induction hypothesis (1) we have $G' \hleq X^L$. Thus $G' \htri X$.
				\item [] Case iii) $G^R \hleq X^L$. $G^R$ is a right option of $G'$, so we have $G' \htri X$.\qedhere
			\end{itemize}
		\end{itemize}
	\end{proof}
	
	\begin{theorem}\label{thm:rev-dom-preserves-c}
		Let $G \in \cal{C}$, with $G = \{A, B, G^{\cal{L}}\mid G^{\cal{R}}\} $, and $A^R \hleq B$. Let \\$G' = \{A^{R\Le}, B, G^{\cal{L}}\mid G^{\cal{R}}\}$. Then $G' \in \cal{C}$.
	\end{theorem}
	
	\begin{proof}
		Let $X \in \cal{D}$ with $G'^R \htri X$ for some $G'^R$. We wish to show $G' \htri X$. $G'^R$ is a right option of $G$ so by Definition~\ref{def:classes-c-and-d} we have that $G \htri X$. By Lemma~\ref{thm:rev-dom-G-X}, it follows that $G' \htri X$.
		
		Let $X \in \cal{D}$ with $X \htri G'^L$ for some $G'^L$. We wish to show $X \htri G'$. There are two options for $G'^L$.
		\begin{itemize}
			\item [] Case i) $G'^L = A^{RL}$. Then by Definition~\ref{def:classes-c-and-d} $X \htri A^{RL} \Rightarrow X \htri A^R$. Since $A^R \hleq B$, we have by Theorem~\ref{thm:c-transitivity-D-left} we have that $X\htri B$. Finally, since $B$ is a left option of $G$ it follows from Definition~\ref{def:classes-c-and-d} that $X\htri G$. It follows from Lemma~\ref{thm:rev-dom-X-G} that $X \htri G'$. 
			\item [] Case ii) $G'^L \neq A^{RL}$. Then $G'^L$ is a left option of $G$. It follows from Definition~\ref{def:classes-c-and-d} that $X \htri G$. By Lemma~\ref{thm:rev-dom-X-G}, we have that $X \htri G'$.
		\end{itemize}
		Thus, $G' \in \cal{C}$.
	\end{proof}
	
	\begin{theorem}\label{thm:rev-dom-equal}
		Let $G \in \cal{C}$, with $G = \{A, B, G^{\cal{L}}\mid G^{\cal{R}}\} $, and $A^R \hleq B$. Then $G \heq \{A^{R\Le}, B, G^{\cal{L}}\mid G^{\cal{R}}\}$.
	\end{theorem}
	
	\begin{proof}
		First we show $G \hleq G'$. Take $G^L$. There are two options for $G^L$.
		\begin{itemize}
			\item [] Case 1) $G^L = A$. We have that $A^R \hleq B$ and $B \htri G$. It follows from Theorem~\ref{thm:c-transitivity-D-left} that $A^R \htri G$. Then, by Definition~\ref{def:classes-c-and-d} we have that $A \htri G$. Finally, by Lemma~\ref{thm:rev-dom-X-G} we have that $A \htri G'$.
			\item [] Case 2) $G^L \neq A$. Then $G^L$ is a left option of $G'$ so $G^L \htri G'$.
		\end{itemize}
		Next, take $G'^R$. $G'^R$ is a right option of $G$, so $G\htri G'^R$. Neither $G$ nor $G'$ is atomic, so we have $G \hleq G'$.
		
		Now we show $G' \hleq G$. Take $G'^L$. There are two options for $G'^L$.
		\begin{itemize}
			\item [] Case i) $G'^L = A^{RL}$. Then since $A^R \hleq B$ it follows that $A^{RL} \htri B$. Since $B$ is a left option of $G$, by Definition~\ref{def:classes-c-and-d}, we have that $A^{RL} \htri G$.
			\item [] Case ii) $G'^L \neq A^{RL}$. Then $G'^L$ is a left option of $G$, so $G'^L\htri G$.
		\end{itemize}
		Take $G^R$. $G^R$ is a right option of $G'$, so $G'\htri G^R$. Neither $G$ nor $G'$ is atomic, so we have $G' \hleq G$.
		
		Thus $G \heq G'$.
	\end{proof}
	
	\subsection{Self-Reversible Domination}\label{subsec:self-reversible-domination}
	
	\begin{lemma}\label{thm:self-rev-dom-X-G}
		Let $G \in \cal{C}$, with $G = \{A, G^{\cal{L}}\mid G^{\cal{R}}\}$, and $A^R \hleq A$. Let \\$G' = \{A^R, A^{R\Le}, G^{\cal{L}}\mid G^{\cal{R}}\}$. Then, for any game $X \in \cal{D}$:
		\begin{itemize}
			\item [1)] $X \hleq G \Rightarrow X \hleq G'$.
			\item [1)] $X \htri G \Rightarrow X \htri G'$.
		\end{itemize}
	\end{lemma}
	
	\begin{proof}
		We prove these statements by mutual induction. Take $X \in \cal{D}$.
		\begin{itemize}
			\item [1)] Take $X^L$. $X^L \htri G \Rightarrow X^L \htri G'$ by induction hypothesis (2). Take $G'^R$. $G'^R$ is a right option of $G$, thus because $X \hleq G$, we have $X \htri G'^R$. $G'$ is non atomic, so we have $X \hleq G'$.
			\item [2)] There are three ways we can have $X \htri G$.
			\begin{itemize}
				\item [] Case i) $X^R \hleq G$. Then by induction hypothesis (1) we have $X^R \hleq G'$. Thus $X \htri G'$.
				\item [] Case ii) $X \hleq G^L$. There are two options for $G^L$.
				\begin{itemize}
					\item [] Case a) $G^L = A$. Then we have $X \htri A^R$. There are three ways this is possible.
					\begin{itemize}
						\item [] Case I) $X^R \hleq A^R$. Since $A^R$ is a left option of $G'$ it follows that $X \htri G'$.
						\item [] Case II) $X \hleq A^{RL}$. $A^{RL}$ is a left option of $G'$ so it follows that $X \htri G'$.
						\item [] Case III) $X^R \hleq A^{RL}$. $A^{RL}$ is a left option of $G'$ so it follows that $X \htri G'$.
					\end{itemize}
					\item [] Case b) $G^L \neq A$. Then $G^L$ is a left option of $G'$. Thus $X \htri G'$.
				\end{itemize}
				\item [] Case iii) $X^R \hleq G^L$. There are two options for $G^L$.
				\begin{itemize}
					\item [] Case a) $G^L = A$. Then we have $X^R \htri A^R$. It follows from Definition~\ref{def:classes-c-and-d} that $X \htri A^R$. There are three ways this is possible.
					\begin{itemize}
						\item [] Case I) $X^R \hleq A^R$. Since $A^R$ is a left option of $G'$ it follows that $X \htri G'$.
						\item [] Case II) $X \hleq A^{RL}$. $A^{RL}$ is a left option of $G'$ so it follows that $X \htri G'$.
						\item [] Case III) $X^R \hleq A^{RL}$. $A^{RL}$ is a left option of $G'$ so it follows that $X \htri G'$.
					\end{itemize}
					\item [] Case b) $G^L \neq A$. Then $G^L$ is a left option of $G'$. Thus $X \htri G'$.\qedhere
				\end{itemize}
			\end{itemize}
		\end{itemize}
	\end{proof}
	
	\begin{lemma}\label{thm:self-rev-dom-G-X}
		Let $G \in \cal{C}$, with $G = \{A, G^{\cal{L}}\mid G^{\cal{R}}\}$, and $A^R \hleq A$. Let \\$G' = \{A^R, A^{R\Le}, G^{\cal{L}}\mid G^{\cal{R}}\}$. Then, for any game $X \in \cal{D}$:
		\begin{itemize}
			\item [1)] $G \hleq X \Rightarrow G' \hleq X$.
			\item [1)] $G \htri X \Rightarrow G' \htri X$.
		\end{itemize}
	\end{lemma}
	
	\begin{proof}
		We prove these statements by mutual induction. Take $X \in \cal{D}$.
		\begin{itemize}
			\item [1)] Take $G'^L$. There are three options for $G'^L$.
			\begin{itemize}
				\item [] Case i) $G'^L = A^R$. $A \htri G$ since $A$ is a left option of $G$. Since $A^R \hleq A$, and $G \hleq X$ it follows from Theorem~\ref{thm:c-transitivity-D-left} that $A^R \htri X$. 
				\item [] Case ii) $G'^L = A^{RL}$. Since $A^R \hleq A$ we have that $A^{RL} \htri A$. Since $A$ is a left option of $G$, it follows from Definition~\ref{def:classes-c-and-d} that $A^{RL} \htri G$. Since $G \hleq X$, it then follows from Theorem~\ref{thm:c-transitivity-D-left} that $A^{RL} \htri X$.
				\item [] Case iii) $G'^L \neq A^R$ and $G'^L \neq A^{RL}$. Then $G'^L$ is a left option of $G$. Because $G \hleq X$, it follows that $G'^L \htri X$.
			\end{itemize}
			Next, take $X^R$. $G \htri X^R \Rightarrow G' \htri X^R$ by induction hypothesis (2). $G'$ is non atomic, so we have $G' \hleq X$.
			\item [2)] There are three ways we can have $G \htri X$. 
			\begin{itemize}
				\item [] Case i) $G^R \hleq X$. $G^R$ is a right option of $G'$, so we have $G' \htri X$.
				\item [] Case ii) $G \hleq X^L$. Then by induction hypothesis (1) we have $G' \hleq X^L$. Thus $G' \htri X$.
				\item [] Case iii) $G^R \hleq X^L$. $G^R$ is a right option of $G'$, so we have $G' \htri X$.\qedhere
			\end{itemize}
		\end{itemize}
	\end{proof}
	
	\begin{theorem}\label{thm:self-rev-dom-preserves-c}
		Let $G \in \cal{C}$, with $G = \{A, G^{\cal{L}}\mid G^{\cal{R}}\}$, and $A^R \hleq A$. Let \\$G' = \{A^R, A^{R\Le}, G^{\cal{L}}\mid G^{\cal{R}}\}$. Then $G' \in \cal{C}$.
	\end{theorem}
	
	\begin{proof}
		Let $X \in \cal{D}$ with $G'^R \htri X$ for some $G'^R$. We wish to show $G' \htri X$.
		
		$G'^R$ is a right option of $G$. By Definition~\ref{def:classes-c-and-d} we have that $G \htri X$. By Lemma~\ref{thm:self-rev-dom-G-X}, it follows that $G' \htri X$.
		
		Let $X \in \cal{D}$ with $X \htri G'^L$ for some $G'^L$. We wish to show $X \htri G'$. There are three options for $G'^L$.
		\begin{itemize}
			\item [] Case i) $G'^L = A^R$. We have that $X \htri A^R \hleq A$. It follows from Theorem~\ref{thm:c-transitivity-D-left} that $X \htri A$. Then, since $A$ is a left option of $G$, it follows from Definition~\ref{def:classes-c-and-d} that $X \htri G$. By Lemma~\ref{thm:self-rev-dom-X-G} we have that $X \htri G'$. 
			\item [] Case ii) $G'^L = A^{RL}$. By Definition~\ref{def:classes-c-and-d} $X \htri A^{RL} \Rightarrow X \htri A^R$. Then, since $A^R \hleq A$ we have by Theorem~\ref{thm:c-transitivity-D-left} that $X \htri A$. Finally, since $A$ is a left option of $G$, it follows from Definition~\ref{def:classes-c-and-d} that $X\htri G$. By Lemma~\ref{thm:self-rev-dom-X-G} we have that $X \htri G'$. 
			\item [] Case iii) $G'^L \neq A^R$ and $G'^L \neq A^{RL}$. Then $G'^L$ is a left option of $G$. It follows from Definition~\ref{def:classes-c-and-d} that $X \htri G$. By Lemma~\ref{thm:self-rev-dom-X-G}, it follows that $X \htri G'$.
		\end{itemize}
		Thus, $G' \in \cal{C}$.
	\end{proof}
	
	\begin{theorem}\label{thm:self-rev-dom-equal}
		Let $G \in \cal{C}$, with $G = \{A, G^{\cal{L}}\mid G^{\cal{R}}\} $, and $A^R \hleq A$. Then \\$G \heq \{A^R, A^{R\Le}, G^{\cal{L}}\mid G^{\cal{R}}\}$.
	\end{theorem}
	
	\begin{proof}
		First we show $G \hleq G'$. Take $G^L$. There are two options for what $G^L$ could be.
		\begin{itemize}
			\item [] Case i) $G^L = A$. Then $A \htri G$ since $A$ is a left option of $G$. It follows from Lemma~\ref{thm:self-rev-dom-X-G} that $A \htri G'$.
			\item [] Case ii) $G^L \neq A$. Then $G^L$ is a left option of $G'$ so $G^L \htri G'$.
		\end{itemize}
		Next, take $G'^R$. $G'^R$ is a right option of $G$, so $G\htri G'^R$. Neither $G$ nor $G'$ is atomic, so we have $G \hleq G'$.
		
		Now we show $G' \hleq G$. Take $G'^L$. There are three options for $G'^L$.
		\begin{itemize}
			\item [] Case i) $G'^L = A^R$. We have that $A^R \hleq A$. Since $A$ is a left option of $G$ it follows that $A^R \htri G$.
			\item [] Case ii) $G'^L = A^{RL}$. Then, since $A^R \hleq A$ it follows that $A^{RL} \htri A$. It follows from Definition~\ref{def:classes-c-and-d} that $A^{RL} \htri G$.
			\item [] Case iii) $G'^L \neq A^R$ and $G'^L \neq A^{RL}$. Then $G'^L$ is a left option of $G$, so $G'^L\htri G$.
		\end{itemize}
		Next, take $G^R$. $G^R$ is a right option of $G'$, so $G'\htri G^R$. Neither $G$ nor $G'$ is atomic, so we have $G' \hleq G$. 
		
		Thus $G \heq G'$.
	\end{proof}
	
	\subsection{Two-Reversibility}\label{subsec:two-reversibility}
	
	\begin{lemma}\label{thm:two-reversibility-X-G}
		Let $G \in \cal{C}$, with $G = \{A,~G^{\cal{L}}\mid~G^{\cal{R}}\} $, and $A^{RR} \hleq G$. Moreover, let $G'~=~\{A^{RR\Le},~G^{\cal{L}}\mid~G^{\cal{R}}\}$. Then, for any game $X \in \cal{D}$:
		\begin{itemize}
			\item [1)] $X \hleq G \Rightarrow X \hleq G'$.
			\item [1)] $X \htri G \Rightarrow X \htri G'$.
		\end{itemize}
	\end{lemma}
	
	\begin{proof}
		We prove these statements by mutual induction. Take $X \in \cal{D}$.
		\begin{itemize}
			\item [1)] Take $X^L$. $X^L \htri G \Rightarrow X^L \htri G'$ by induction hypothesis (2). Take $G'^R$. $G'^R$ is a right option of $G$, thus because $X \hleq G$, we have $X \htri G'^R$. $G'$ is non atomic, so we have $X \hleq G'$.
			\item [2)] There are three ways we can have $X \htri G$.
			\begin{itemize}
				\item [] Case i) $X^R \hleq G$. Then by induction hypothesis (1) we have $X^R \hleq G'$. Thus $X \htri G'$.
				\item [] Case ii) $X \hleq G^L$. Then there are two options for what $G^L$ could be.
				\begin{itemize}
					\item [] Case a) $G^L = A$. Since $A^R \htri A^{RR}$ we have that $A \htri A^{RR}$. Then it follows that $X \hleq A \htri A^{RR}$. By Theorem~\ref{thm:c-transitivity-D-left} we have that $X \htri A^{RR}$. There are three ways this is possible.
					\begin{itemize}
						\item [] Case I) $X^R \hleq A^{RR}$. But $A^{RR} \hleq G$. By Theorem~\ref{thm:c-transitivity-D-left} we have that $X^R \hleq G$. By induction hypothesis (1), we then have $X^R \hleq G'$. Thus $X \htri G'$.
						\item [] Case II) $X \hleq A^{RRL}$. $A^{RRL}$ is a left option of $G'$ so it follows that $X \htri G'$.
						\item [] Case III) $X^R \hleq A^{RRL}$. $A^{RRL}$ is a left option of $G'$ so it follows that $X \htri G'$.
					\end{itemize}
					\item [] Case b) $G^L \neq A$. Then $G^L$ is a left option of $G'$. Thus $X \htri G'$.
				\end{itemize}
				\item [] Case iii) $X^R \hleq G^L$. There are again two options for what $G^L$ could be.
				\begin{itemize}
					\item [] Case a) $G^L = A$. Since $A^R \htri A^{RR}$ it follows that $A \htri A^{RR}$. Then $X^R \hleq A \htri A^{RR}$. By Theorem~\ref{thm:c-transitivity-D-left} we have that $X^R \htri A^{RR}$. It follows from Definition~\ref{def:classes-c-and-d} that $X \htri A^{RR}$. There are three ways this is possible
					\begin{itemize}
						\item [] Case I) $X^R \hleq A^{RR}$. But $A^{RR} \hleq G$. By Theorem~\ref{thm:c-transitivity-D-left} we have that $X^R \hleq G$. By induction hypothesis (1), we then have $X^R \hleq G'$. Thus $X \htri G'$.
						\item [] Case II) $X \hleq A^{RRL}$. $A^{RRL}$ is a left option of $G'$ so it follows that $X \htri G'$.
						\item [] Case III) $X^R \hleq A^{RRL}$. $A^{RRL}$ is a left option of $G'$ so it follows that $X \htri G'$.
					\end{itemize}
					\item [] Case b) $G^L \neq A$. Then $G^L$ is a left option of $G'$. Thus $X \htri G'$.\qedhere
				\end{itemize}
			\end{itemize}
		\end{itemize}
	\end{proof}
	
	\begin{lemma}\label{thm:two-reversibility-G-X}
		Let $G \in \cal{C}$, with $G = \{A, G^{\cal{L}}\mid G^{\cal{R}}\} $, and $A^{RR} \hleq G$. Let \\$G' = \{A^{RR\Le}, G^{\cal{L}}\mid G^{\cal{R}}\} $. Then, for any game $X \in \cal{D}$:
		\begin{itemize}
			\item [1)] $G \hleq X \Rightarrow G' \hleq X$.
			\item [1)] $G \htri X \Rightarrow G' \htri X$.
		\end{itemize}
	\end{lemma}
	
	\begin{proof}
		We prove these statements by mutual induction. Take $X \in \cal{D}$.
		\begin{itemize}
			\item [1)] Take $G'^L$. There are two options for what $G'^L$ could be.
			\begin{itemize}
				\item [] Case i) $G'^L = A^{RRL}$. Since $A^{RR} \hleq G$ it follows that $A^{RRL} \htri G \hleq X$. By Theorem~\ref{thm:c-transitivity-D-right}, we have that $A^{RRL} \htri X$.
				\item [] Case ii) $G'^L \neq A^{RRL}$. Then $G'^L$ is a left option of $G$. Because $G \hleq X$, it follows that $G'^L \htri X$.
			\end{itemize}
			Now take $X^R$. $G \htri X^R \Rightarrow G' \htri X^R$ by induction hypothesis (2). $G'$ is non atomic, so we have $G' \hleq X$.
			\item [2)] There are three ways we can have $G \htri X$.
			\begin{itemize}
				\item [] Case i) $G^R \hleq X$. $G^R$ is a right option of $G'$, so we have $G' \htri X$.
				\item [] Case ii) $G \hleq X^L$. Then by induction hypothesis (1) we have $G' \hleq X^L$. Thus $G' \htri X$.
				\item [] Case iii) $G^R \hleq X^L$. $G^R$ is a right option of $G'$, so we have $G' \htri X$.\qedhere
			\end{itemize}
		\end{itemize}
	\end{proof}
	
	\begin{theorem}\label{thm:two-reversiblity-preserves-c}
		Let $G \in \cal{C}$, with $G = \{A, G^{\cal{L}}\mid G^{\cal{R}}\} $, and $A^{RR} \hleq G$. Let \\$G' = \{A^{RR\Le}, G^{\cal{L}}\mid G^{\cal{R}}\}$. Then $G' \in \cal{C}$.
	\end{theorem}
	
	\begin{proof}
		Let $X \in \cal{D}$ with $G'^R \htri X$ for some $G'^R$. We wish to show $G' \htri X$. First, we note that $G'^R$ is a right option of $G$. By Theorem~\ref{thm:c-transitivity-D-right} we then have that $G \htri X$. By Lemma~\ref{thm:two-reversibility-G-X}, it follows that $G' \htri X$.
		
		Let $X \in \cal{D}$ with $X \htri G'^L$ for some $G'^L$. We wish to show $X \htri G'$. There are two options for what $G'^L$ is.
		\begin{itemize}
			\item [] Case i) $G'^L = A^{RRL}$. By Definition~\ref{def:classes-c-and-d} $X \htri A^{RR}$. Since $A^{RR} \hleq G$, we have that $X \htri G$ by Theorem~\ref{thm:c-transitivity-D-left}. It follows from Lemma~\ref{thm:two-reversibility-X-G} that $X \htri G'$. 
			\item [] Case ii) $G'^L \neq A^{RL}$. Then $G'^L$ is a left option of $G$. Since $G \in \cal{C}$, by Definition~\ref{def:classes-c-and-d} we have that $X \htri G$. By Lemma~\ref{thm:two-reversibility-X-G}, it follows that $X \htri G'$.
		\end{itemize}
		Thus, $G' \in \cal{C}$.
	\end{proof}
	
	\begin{theorem}\label{thm:two-reversibility-equal}
		Let $G \in \cal{C}$, with $G = \{A, G^{\cal{L}}\mid G^{\cal{R}}\} $, and $A^{RR} \hleq G$. Then \\$G \heq \{A^{RR\Le}, G^{\cal{L}}\mid G^{\cal{R}}\} $.
	\end{theorem}
	
	\begin{proof}
		First we show $G \hleq G'$. Take $G^L$. There are two options for what $G^L$ can be.
		\begin{itemize}
			\item [] Case 1) $G^L = A$. Then since $A^{RR} \hleq G$ we have that $A^{RR} \hleq G'$ by Lemma~\ref{thm:reversibility-X-G}. Thus $A^R \htri G'$ by Definition~\ref{def:classes-c-and-d}. It then follows from Definition~\ref{def:classes-c-and-d} that $A \htri G'$.
			\item [] Case 2) $G^L \neq A$. Then $G^L$ is a left option of $G'$ so $G^L \htri G'$.
		\end{itemize}
		Next, take $G'^R$. $G'^R$ is a right option of $G$, so $G\htri G'^R$. Neither $G$ nor $G'$ is atomic, so we have $G \hleq G'$.
		
		Now we show $G' \hleq G$. Take $G'^L$. There are two options for what $G'^L$ can be.
		\begin{itemize}
			\item [] Case i) $G'^L = A^{RRL}$. Then since $A^{RR} \hleq G$ it follows that $A^{RRL} \htri G$.
			\item [] Case ii) $G'^L \neq A^{RRL}$. Then $G'^L$ is a left option of $G$, so $G'^L\htri G$.
		\end{itemize}
		Next, take $G^R$. $G^R$ is a right option of $G'$, so $G'\htri G^R$. Neither $G$ nor $G'$ is atomic, so we have $G' \hleq G$. 
		
		Thus $G \heq G'$.
	\end{proof}
	
	\subsection{Two-Reversible Domination}\label{subsec:two-reversible-domination}
	
	\begin{lemma}\label{thm:two-rev-dom-X-G}
		Let $G \in \cal{C}$, with $G = \{A, B, G^{\cal{L}}\mid G^{\cal{R}}\} $, and $A^{RR} \hleq B$. Let \\$G' = \{A^{RR\Le}, B, G^{\cal{L}}\mid G^{\cal{R}}\}$. Then, for any game $X \in \cal{D}$:
		\begin{itemize}
			\item [1)] $X \hleq G \Rightarrow X \hleq G'$.
			\item [1)] $X \htri G \Rightarrow X \htri G'$.
		\end{itemize}
	\end{lemma}
	
	\begin{proof}
		We prove these statements by mutual induction. Take $X \in \cal{D}$.
		\begin{itemize}
			\item [1)] Take $X^L$. $X^L \htri G \Rightarrow X^L \htri G'$ by induction hypothesis (2). Take $G'^R$. $G'^R$ is a right option of $G$, so we have $X \htri G'^R$. $G'$ is non atomic, so we have $X \hleq G'$.
			\item [2)] There are three ways we can have $X \htri G$.
			\begin{itemize}
				\item [] Case i) $X^R \hleq G$. Then by induction hypothesis (1) we have $X^R \hleq G'$. Thus $X \htri G'$.
				\item [] Case ii) $X \hleq G^L$. Then there are two options for $G^L$.
				\begin{itemize}
					\item [] Case a) $G^L = A$. Since $A^R \htri A^{RR}$ it follows that $A \htri A^{RR}$. Then we have $X \hleq A \htri A^{RR}$. By Theorem~\ref{thm:c-transitivity-D-left} it follows that $X \htri A^{RR}$. There are three ways this is possible
					\begin{itemize}
						\item [] Case I) $X^R \hleq A^{RR}$. But $A^{RR} \hleq B$. By Theorem~\ref{thm:c-transitivity-D-left} we have that $X^R \hleq B$. Since $B$ is a left option of $G'$ it follows that $X \htri G'$.
						\item [] Case II) $X \hleq A^{RRL}$. $A^{RRL}$ is a left option of $G'$ so it follows that $X \htri G'$.
						\item [] Case III) $X^R \hleq A^{RRL}$. $A^{RRL}$ is a left option of $G'$ so it follows that $X \htri G'$.
					\end{itemize}
					\item [] Case b) $G^L \neq A$. Then $G^L$ is a left option of $G'$. Thus $X \htri G'$.
				\end{itemize}
				\item [] Case iii) $X^R \hleq G^L$. There are again two options for $G^L$.
				\begin{itemize}
					\item [] Case a) $G^L = A$. Since $A^R \htri A^{RR}$ it follows that $A \htri A^{RR}$. Then we have $X^R \hleq A \htri A^{RR}$. By Theorem~\ref{thm:c-transitivity-D-left} it follows that $X^R \htri A^{RR}$. By Definition~\ref{def:classes-c-and-d} it follows that $X \htri A^{RR}$. There are three ways this is possible.
					\begin{itemize}
						\item [] Case I) $X^R \hleq A^{RR}$. But $A^{RR} \hleq B$. By Theorem~\ref{thm:c-transitivity-D-left} we have that $X^R \hleq B$. Since $B$ is a left option of $G'$ it follows that $X \htri G'$.
						\item [] Case II) $X \hleq A^{RRL}$. $A^{RRL}$ is a left option of $G'$ so it follows that $X \htri G'$.
						\item [] Case III) $X^R \hleq A^{RRL}$. $A^{RRL}$ is a left option of $G'$ so it follows that $X \htri G'$.
					\end{itemize}
					\item [] Case b) $G^L \neq A$. Then $G^L$ is a left option of $G'$. Thus $X \htri G'$.\qedhere
				\end{itemize}
			\end{itemize}
		\end{itemize}
	\end{proof}
	
	\begin{lemma}\label{thm:two-rev-dom-G-X}
		Let $G \in \cal{C}$, with $G = \{A, B, G^{\cal{L}}\mid G^{\cal{R}}\} $, and $A^{RR} \hleq B$. Let \\$G' = \{A^{RR\Le}, B, G^{\cal{L}}\mid G^{\cal{R}}\}$. Then, for any game $X \in \cal{D}$:
		\begin{itemize}
			\item [1)] $G \hleq X \Rightarrow G' \hleq X$.
			\item [1)] $G \htri X \Rightarrow G' \htri X$.
		\end{itemize}
	\end{lemma}
	
	\begin{proof}
		We prove these statements by mutual induction. Take $X \in \cal{D}$.
		\begin{itemize}
			\item [1)] Take $G'^L$. There are two possible options for $G'^L$.
			\begin{itemize}
				\item [] Case i) $G'^L = A^{RRL}$. Since $A^{RR} \hleq B$ we have that $A^{RRL} \htri B$. Since $B$ is a left option of $G$ it follows from Definition~\ref{def:classes-c-and-d} that $A^{RRL} \htri G$. But $G \hleq X$. By Theorem~\ref{thm:c-transitivity-D-right} we have that $A^{RRL} \htri X$.
				\item [] Case ii) $G'^L \neq A^{RRL}$. Then $G'^L$ is a left option of $G$. Because $G \hleq X$, it follows that $G'^L \htri X$.
			\end{itemize}
			Next, take $X^R$. $G \htri X^R \Rightarrow G' \htri X^R$ by induction hypothesis (2). $G'$ is non atomic, so we have $G' \hleq X$.
			\item [2)] There are three ways we can have $G \htri X$.
			\begin{itemize}
				\item [] Case i) $G^R \hleq X$. $G^R$ is a right option of $G'$, so we have $G' \htri X$.
				\item [] Case ii) $G \hleq X^L$. Then by induction hypothesis (1) we have $G' \hleq X^L$. Thus $G' \htri X$.
				\item [] Case iii) $G^R \hleq X^L$. $G^R$ is a right option of $G'$, so we have $G' \htri X$.\qedhere
			\end{itemize}
		\end{itemize}
	\end{proof}
	
	\begin{theorem}\label{thm:two-rev-dom-preserves-c}
		Let $G \in \cal{C}$, with $G = \{A, B, G^{\cal{L}}\mid G^{\cal{R}}\} $, and $A^{RR} \hleq B$. Let \\$G' = \{A^{RR\Le}, B, G^{\cal{L}}\mid G^{\cal{R}}\}$. Then $G' \in \cal{C}$.
	\end{theorem}
	
	\begin{proof}
		Let $X \in \cal{D}$ with $G'^R \htri X$ for some $G'^R$. We wish to show $G' \htri X$. $G'^R$ is a right option of $G$ so by Definition~\ref{def:classes-c-and-d} we have that $G \htri X$. By Lemma~\ref{thm:two-rev-dom-G-X}, it follows that $G' \htri X$.
		
		Let $X \in \cal{D}$ with $X \htri G'^L$ for some $G'^L$. We wish to show $X \htri G'$. There are two options for $G'^L$.
		\begin{itemize}
			\item [] Case i) $G'^L = A^{RRL}$. Then by Definition~\ref{def:classes-c-and-d} $X \htri A^{RRL} \Rightarrow X \htri A^{RR}$. Since $A^{RR} \hleq B$, we have by Theorem~\ref{thm:c-transitivity-D-left} that $X \htri B$. Finally, since $B$ is a left option of $G$ it follows from Definition~\ref{def:classes-c-and-d} that $X \htri G$. It follows from Lemma~\ref{thm:two-rev-dom-X-G} that $X \htri G'$. 
			\item [] Case ii) $G'^L \neq A^{RRL}$. Then $G'^L$ is a left option of $G$. It follows from Definition~\ref{def:classes-c-and-d} that $X \htri G$. By Lemma~\ref{thm:two-rev-dom-X-G}, we have that $X \htri G'$.
		\end{itemize}
		Thus, $G' \in \cal{C}$.
	\end{proof}
	
	\begin{theorem}\label{thm:two-rev-dom-equal}
		Let $G \in \cal{C}$, with $G = \{A, B, G^{\cal{L}}\mid G^{\cal{R}}\} $, and $A^{RR} \hleq B$. Then $G \heq \{A^{RR\Le}, B, G^{\cal{L}}\mid G^{\cal{R}}\}$.
	\end{theorem}
	
	\begin{proof}
		First we show $G \hleq G'$. Take $G^L$. There are two options for $G^L$.
		\begin{itemize}
			\item [] Case i) $G^L = A$. We have that $A^{RR} \hleq B$ and $B \htri G$. It follows from Theorem~\ref{thm:c-transitivity-D-left} that $A^{RR} \htri G$. Then, by Definition~\ref{def:classes-c-and-d} we have that $A \htri G$. Finally, by Lemma~\ref{thm:two-rev-dom-X-G} we have that $A \htri G'$.
			\item [] Case ii) $G^L \neq A$. Then $G^L$ is a left option of $G'$ so $G^L \htri G'$.
		\end{itemize}
		Next, take $G'^R$. $G'^R$ is a right option of $G$, so $G\htri G'^R$. Neither $G$ nor $G'$ is atomic, so we have $G \hleq G'$.
		
		Now we show $G' \hleq G$. Take $G'^L$. There are two options for $G'^L$.
		\begin{itemize}
			\item [] Case i) $G'^L = A^{RRL}$. Then since $A^{RR} \hleq B$ it follows that $A^{RRL} \htri B$. Since $B$ is a left option of $G$, by Definition~\ref{def:classes-c-and-d}, we have that $A^{RRL} \htri G$.
			\item [] Case ii) $G'^L \neq A^{RRL}$. Then $G'^L$ is a left option of $G$, so $G'^L\htri G$.
		\end{itemize}
		Take $G^R$. $G^R$ is a right option of $G'$, so $G'\htri G^R$. Neither $G$ nor $G'$ is atomic, so we have $G' \hleq G$.
		
		Thus $G \heq G'$.
	\end{proof}
	
	\subsection{Two-Self-Reversible Domination}\label{subsec:two-self-reversible-domination}
	
	\begin{lemma}\label{thm:two-self-rev-dom-X-G}
		Let $G \in \cal{C}$, with $G = \{A, G^{\cal{L}}\mid G^{\cal{R}}\}$, and $A^{RR} \hleq A$. Let \\$G' = \{A^{RR}, A^{RR\Le}, G^{\cal{L}}\mid G^{\cal{R}}\}$. Then, for any game $X \in \cal{D}$:
		\begin{itemize}
			\item [1)] $X \hleq G \Rightarrow X \hleq G'$.
			\item [1)] $X \htri G \Rightarrow X \htri G'$.
		\end{itemize}
	\end{lemma}
	
	\begin{proof}
		We prove these statements by mutual induction. Take $X \in \cal{D}$.
		\begin{itemize}
			\item [1)] Take $X^L$. $X^L \htri G \Rightarrow X^L \htri G'$ by induction hypothesis (2). Take $G'^R$. $G'^R$ is a right option of $G$, thus because $X \hleq G$, we have $X \htri G'^R$. $G'$ is non atomic, so we have $X \hleq G'$.
			\item [2)] There are three ways we can have $X \htri G$.
			\begin{itemize}
				\item [] Case i) $X^R \hleq G$. Then by induction hypothesis (1) we have $X^R \hleq G'$. Thus $X \htri G'$.
				\item [] Case ii) $X \hleq G^L$. There are two options for $G^L$.
				\begin{itemize}
					\item [] Case a) $G^L = A$. Since $A^R \htri A^{RR}$ it follows that $A \htri A^{RR}$. Then we have $X \hleq A \htri A^{RR}$. It follows from Theorem~\ref{thm:c-transitivity-D-left} that $X \htri A^{RR}$. There are three ways this is possible.
					\begin{itemize}
						\item [] Case I) $X^R \hleq A^{RR}$. Since $A^{RR}$ is a left option of $G'$ it follows that $X \htri G'$.
						\item [] Case II) $X \hleq A^{RRL}$. $A^{RRL}$ is a left option of $G'$ so it follows that $X \htri G'$.
						\item [] Case III) $X^R \hleq A^{RRL}$. $A^{RRL}$ is a left option of $G'$ so it follows that $X \htri G'$.
					\end{itemize}
					\item [] Case b) $G^L \neq A$. Then $G^L$ is a left option of $G'$. Thus $X \htri G'$.
				\end{itemize}
				\item [] Case iii) $X^R \hleq G^L$. There are two options for $G^L$.
				\begin{itemize}
					\item [] Case a) $G^L = A$. Since $A^R \htri A^{RR}$ it follows that $A \htri A^{RR}$. Then we have $X^R \hleq A \htri A^{RR}$. It follows from Theorem~\ref{thm:c-transitivity-D-left} that $X^R \htri A^{RR}$. By Definition~\ref{def:classes-c-and-d} we have that $X \htri A^{RR}$. There are three ways this is possible.
					\begin{itemize}
						\item [] Case I) $X^R \hleq A^{RR}$. Since $A^{RR}$ is a left option of $G'$ it follows that $X \htri G'$.
						\item [] Case II) $X \hleq A^{RRL}$. $A^{RRL}$ is a left option of $G'$ so it follows that $X \htri G'$.
						\item [] Case III) $X^R \hleq A^{RL}$. $A^{RRL}$ is a left option of $G'$ so it follows that $X \htri G'$.
					\end{itemize}
					\item [] Case b) $G^L \neq A$. Then $G^L$ is a left option of $G'$. Thus $X \htri G'$.\qedhere
				\end{itemize}
			\end{itemize}
		\end{itemize}
	\end{proof}
	
	\begin{lemma}\label{thm:two-self-rev-dom-G-X}
		Let $G \in \cal{C}$, with $G = \{A, G^{\cal{L}}\mid G^{\cal{R}}\}$, and $A^{RR} \hleq A$. Let \\$G' = \{A^{RR}, A^{RR\Le}, G^{\cal{L}}\mid G^{\cal{R}}\}$. Then, for any game $X \in \cal{D}$:
		\begin{itemize}
			\item [1)] $G \hleq X \Rightarrow G' \hleq X$.
			\item [1)] $G \htri X \Rightarrow G' \htri X$.
		\end{itemize}
	\end{lemma}
	
	\begin{proof}
		We prove these statements by mutual induction. Take $X \in \cal{D}$.
		\begin{itemize}
			\item [1)] Take $G'^L$. There are three options for $G'^L$.
			\begin{itemize}
				\item [] Case i) $G'^L = A^{RR}$. $A \htri G$ since $A$ is a left option of $G$. Since $A^{RR} \hleq A$, and $G \hleq X$ it follows from Theorem~\ref{thm:c-transitivity-D-left} that $A^{RR} \htri X$. 
				\item [] Case ii) $G'^L = A^{RRL}$. Since $A^{RR} \hleq A$ we have that $A^{RRL} \htri A$. Since $A$ is a left option of $G$, it follows from Definition~\ref{def:classes-c-and-d} that $A^{RRL} \htri G$. Since $G \hleq X$, it then follows from Theorem~\ref{thm:c-transitivity-D-left} that $A^{RRL} \htri X$.
				\item [] Case iii) $G'^L \neq A^{RR}$ and $G'^L \neq A^{RRL}$. Then $G'^L$ is a left option of $G$. Because $G \hleq X$, it follows that $G'^L \htri X$.
			\end{itemize}
			Next, take $X^R$. $G \htri X^R \Rightarrow G' \htri X^R$ by induction hypothesis (2). $G'$ is non atomic, so we have $G' \hleq X$.
			\item [2)] There are three ways we can have $G \htri X$. 
			\begin{itemize}
				\item [] Case i) $G^R \hleq X$. $G^R$ is a right option of $G'$, so we have $G' \htri X$.
				\item [] Case ii) $G \hleq X^L$. Then by induction hypothesis (1) we have $G' \hleq X^L$. Thus $G' \htri X$.
				\item [] Case iii) $G^R \hleq X^L$. $G^R$ is a right option of $G'$, so we have $G' \htri X$.\qedhere
			\end{itemize}
		\end{itemize}
	\end{proof}
	
	\begin{theorem}\label{thm:two-self-rev-dom-preserves-c}
		Let $G \in \cal{C}$, with $G = \{A, G^{\cal{L}}\mid G^{\cal{R}}\}$, and $A^{RR} \hleq A$. Let \\$G' = \{A^{RR}, A^{RR\Le}, G^{\cal{L}}\mid G^{\cal{R}}\}$. Then $G' \in \cal{C}$.
	\end{theorem}
	
	\begin{proof}
		Let $X \in \cal{D}$ with $G'^R \htri X$ for some $G'^R$. We wish to show $G' \htri X$.
		
		$G'^R$ is a right option of $G$. By Definition~\ref{def:classes-c-and-d} we have that $G \htri X$. By Lemma~\ref{thm:two-self-rev-dom-G-X}, it follows that $G' \htri X$.
		
		Let $X \in \cal{D}$ with $X \htri G'^L$ for some $G'^L$. We wish to show $X \htri G'$. There are three options for $G'^L$.
		\begin{itemize}
			\item [] Case i) $G'^L = A^{RR}$. We have that $X \htri A^{RR} \hleq A$. It follows from Theorem~\ref{thm:c-transitivity-D-left} that $X \htri A$. Then, since $A$ is a left option of $G$, it follows from Definition~\ref{def:classes-c-and-d} that $X \htri G$. By Lemma~\ref{thm:two-self-rev-dom-X-G} we have that $X \htri G'$. 
			\item [] Case ii) $G'^L = A^{RRL}$. By Definition~\ref{def:classes-c-and-d} $X \htri A^{RRL} \Rightarrow X \htri A^{RR}$. Then, since $A^{RR} \hleq A$ we have by Theorem~\ref{thm:c-transitivity-D-left} that $X \htri A$. Finally, since $A$ is a left option of $G$, it follows from Definition~\ref{def:classes-c-and-d} that $X\htri G$. By Lemma~\ref{thm:two-self-rev-dom-X-G} we have that $X \htri G'$. 
			\item [] Case iii) $G'^L \neq A^{RR}$ and $G'^L \neq A^{RRL}$. Then $G'^L$ is a left option of $G$. It follows from Definition~\ref{def:classes-c-and-d} that $X \htri G$. By Lemma~\ref{thm:two-self-rev-dom-X-G}, it follows that $X \htri G'$.
		\end{itemize}
		Thus, $G' \in \cal{C}$.
	\end{proof}
	
	\begin{theorem}\label{thm:two-self-rev-dom-equal}
		Let $G \in \cal{C}$, with $G = \{A, G^{\cal{L}}\mid G^{\cal{R}}\} $, and $A^{RR} \hleq A$. Then \\$G \heq \{A^{RR}, A^{RR\Le}, G^{\cal{L}}\mid G^{\cal{R}}\}$.
	\end{theorem}
	
	\begin{proof}
		First we show $G \hleq G'$. Take $G^L$. 
		\begin{itemize}
			\item [] Case i) $G^L = A$. Then $A \htri G$ since $A$ is a left option of $G$. Since $A^{RR} \hleq A$, it follows from Theorem~\ref{thm:c-transitivity-D-left} that $A^{RR} \htri G$. By Definition~\ref{def:classes-c-and-d} we have that $A \htri G$. Finally, it follows from Lemma~\ref{thm:two-self-rev-dom-X-G} that $A \htri G'$.
			\item [] Case ii) $G^L \neq A$. Then $G^L$ is a left option of $G'$ so $G^L \htri G'$.
		\end{itemize}
		Next, take $G'^R$. $G'^R$ is a right option of $G$, so $G\htri G'^R$. Neither $G$ nor $G'$ is atomic, so we have $G \hleq G'$.
		
		Now we show $G' \hleq G$. Take $G'^L$. There are three options for $G'^L$.
		\begin{itemize}
			\item [] Case i) $G'^L = A^{RR}$. We have that $A^{RR} \hleq A$. Since $A$ is a left option of $G$ it follows from that $A^{RR} \htri G$.
			\item [] Case ii) $G'^L = A^{RRL}$. Then, since $A^{RR} \hleq A$ it follows that $A^{RRL} \htri A$. It follows from Definition~\ref{def:classes-c-and-d} that $A^{RRL} \htri G$.
			\item [] Case iii) $G'^L \neq A^{RR}$ and $G'^L \neq A^{RRL}$. Then $G'^L$ is a left option of $G$, so $G'^L\htri G$.
		\end{itemize}
		Next, take $G^R$. $G^R$ is a right option of $G'$, so $G'\htri G^R$. Neither $G$ nor $G'$ is atomic, so we have $G' \hleq G$. 
		
		Thus $G \heq G'$.
	\end{proof}
	
	\section{Canonical Forms and Uniqueness}
	
	In this section we show that every game in $\cal{C}$ has a unique canonical form and that any two equivalent games in $\cal{C}$ have the same canonical form. First we start by defining what a canonical form is.
	
	\begin{definition}
		We say a game $G$ is in \emph{canonical form} if it cannot be reduced by any of the reductions in 6.1.1-6.1.7, or their duals.
	\end{definition}
	
	Now we show that every finite game in $\C$ has a canonical form which is equivalent to it, and that canonical form is unique.
	
	\begin{theorem}
		 Every (finite) game in $\C$ has a canonical form. 
	\end{theorem}
	
	\begin{proof}
		Let $G$ be a game in $\C$. If $G$ is in canonical form then we are done. If $G$ is not in canonical form, then there is some reduction we can apply to $G$ to get a new game $G'$ which is equivalent to $G$, and still in $\C$. By the induction hypothesis, $G'$ has a canonical form, and that canonical form is equivalent to $G'$ and thus $G$. Note that the reduction rules in~\ref{subsec:domination},~\ref{subsec:reversibility},~\ref{subsec:reversible-domination},~\ref{subsec:two-reversibility}, and~\ref{subsec:two-reversible-domination} decrease the number of nodes but the ones in~\ref{subsec:self-reversible-domination} and~\ref{subsec:two-self-reversible-domination} do not.  Therefore we must ensure that the simplification procedure terminates in a finite number of steps. One way to do that is as follows. Define branch($n$, $G$) to be the number of branches of depth $n$ in the game tree of $G$. Define $G \prec H$ if there exists some $N$ such that branch($n$, $G$) = branch ($n$, $H$) for all $n > N$, and branch($N$, $G$) $<$ branch($N$, $H$). It is easily seen that this is a well-ordering and also that it decreases under each application of a rule from Sections~\ref{subsec:domination}--\ref{subsec:two-self-reversible-domination}. Therefore the process of finding a canonical form must terminate in a finite number of steps.
	\end{proof}
	
	\begin{theorem}
		Let $G, H \in \cal{C}$ be games in canonical form with $G \heq H$, and $G, H \in \cal{C}$. Then $G \equiv H$.
	\end{theorem}
	
	\begin{proof}
		Take $G^L$. Since $G \hleq H$ we have that $G^L \htri H$. There are three ways this can happen.
		\begin{itemize}
			\item [] Case i) $G^{LR} \hleq H$. Then we have $G^{LR} \hleq H \hleq G$. By Theorem~\ref{thm:c-transitivity-D-left} we have that $G^{LR} \hleq G$. But then $G^L$ would be reversible, so we cannot have $G^{LR} \hleq H$.
			\item [] Case ii) $G^{LR} \hleq H^L$. Then we have $G^{LR} \hleq H^L \htri G$. By Theorem~\ref{thm:c-transitivity-D-left} we have that $G^{LR} \htri G$. There are three ways that can happen.
			\begin{itemize}
				\item [] Case a) $G^{LRR} \hleq G$. But then $G^L$ would be two-reversible, so this is not possible.
				\item [] Case b) $G^{LR} \hleq G^L$. But then $G^L$ would be either reversibly dominated or self-reversibly dominated, so this is also not possible.
				\item [] Case c) $G^{LRR} \hleq G^L$. But then $G^L$ would be either two-reversibly dominated or two-self-reversibly dominated, so we cannot have this either.
			\end{itemize}
			Thus we cannot have that $G^{LR} \htri G$, so it is not possible to have $G^{LR} \hleq H^L$. So we must be in case iii.
			\item [] Case iii) $G^L \hleq H^L$. By a similar argument to above, we find that $H^L \hleq G'^L$ for some left option $G'^L$ of $G$. It follows that $G^L \hleq G'^L$. We must then have that $G^L = G'^L$ otherwise $G^L$ would be dominated. Thus $G^L \heq H^L$ for some left option $H^L$ of $H$. By induction hypothesis we have that $G^L \equiv H^L$.
		\end{itemize}
		So we find that every left option $G^L$ of $G$ is equal to some left option of $H$. A similar argument shows the same is true for every left option of $H$, and right option of both $G$ and $H$. Thus $G \equiv H$.
	\end{proof}
	
\chapter{Fallow and Taut Positions}\label{chap:fallow-and-taut}

	Note that in this chapter, unless otherwise stated, every game will be option closed.
	
	\section{Definition}	
	We begin first by revisiting more formally the property of $*$-antimonotonicity that was seen earlier in Chapter~\ref{chap:background}. Recall that $*_d$ is $\gg 0|0\gg$, a game over the 1-element poset ${0}$,
	corresponding to a dead cell.
	
	\begin{definition}
		A game $G$ is called \emph{$*$-antimonotone} if for every left option $G^L$ of $G$ we have $G^L \hs *_d \hleq G$ and for every right option $G^R$ of $G$ we have $G \hleq G^R \hs *_d$.
	\end{definition}
	
	\begin{theorem}\label{thm:rex+-antimonotone}
		Every Rex+ game is $*$-antimonotone.
	\end{theorem}
	
	\begin{proof}
		Let $G$ be a Rex+ position. Take a left option $G^L$ of $G$. We wish to show that $G^L \hs \stard \hleq G$. Take a left option $(G^L \hs \stard)^L$  of $G^L \hs *$. There are three options for what it could be.
		\begin{itemize}
			\item [] Case i) $(G^L \hs \stard)^L = G^L$. Then, $G^L \htri G$.
			\item [] Case ii) $(G^L \hs \stard)^L = G^{LL}$. Then, since $G$ is a Rex+ position it is option closed, thus $G^{LL}$ is a left option of $G$. It follows that $G^{LL} \htri G$.
			\item [] Case iii) $(G^L \hs \stard)^L = G^{LL} \hs \stard$. By induction hypothesis, $G^{LL} \hs \stard \htri G^L$. Since $G$ is option closed, being a Rex+ position, it follows from Theorem~\ref{thm:oc-left-closure} that $G^{LL} \hs *_d \htri G$.
		\end{itemize}
		Next we take a right option $G^R$ of $G$. We wish to show $G^L \hs *_d \htri G^R$. Then, we construct a right option of $G^L \hs *_d$ and a left option of $G^R$. First, let $G^{LR}$ be the right option of $G^L \hs *_d$ which fills the $*_d$, and plays any possible cells that were in the move $G^R$. Then, let $G^{R(L)}$ be the left option of $G^R$ that fills any possible cells that were in the move $G^L$. If no such cells exist then we just take $G^{R(L)} = G^R$. Then, the only difference between $G^{LR}$ and $G^{R(L)}$ is that some black stones may have become white. Since this is never bad for Black, we have that $G^{LR} \hleq G^{R(L)}$. It follows that $G^L \hs \stard \htri G^R$.
		
		Thus we have that $G^L \hs \stard \hleq G$. The proof for right options is dual.
	\end{proof}
	
	We also formalize an observation that was made previously in Chapter~\ref{chap:examples} by proving that two dead cells is equivalent to one dead cell.
	
	\begin{lemma}\label{thm:dead-cells-equal}
		$*_d \hs *_d \heq \stard$.
	\end{lemma}
	
	\begin{proof}
		First we show $*_d \hs *_d \hleq \stard$. Consider a left option of $\stard \hs \stard$. There are two options.
		\begin{itemize}
			\item [] Case i) The left option is $\stard$. Then the right option $[0]$ of $\stard$ is equal to the left option $[0]$ of $\stard$. It follows that $\stard \htri \stard$.
			\item [] Case ii) The left option is $[0]$. Then $[0]$ is equal to the left option $[0]$ of $\stard$. It follows that $[0] \htri \stard$.
		\end{itemize}
		Next, consider a right option of $\stard$. There is only one option, $[0]$. Then $[0]$ is equal to the right option $[0]$ of $*_d \hs *_d$. It follows that $*_d \hs *_d \htri [0]$.
		
		The proof that $\stard \hleq \stard \hs \stard$ follows similarly.
	\end{proof}
	
	Now we recall Definition~\ref{def:outcome-classes}, the definition of an $\N$ position, where an $\N$ position is one where both players have a first player winning strategy. We wish to generalize this notion, to one where both players have a move they wish to make, even if they cannot win in this game. First we recall from Chapter~\ref{chap:examples} that players always want to play in a dead cell $\stard$. We generalize this notion to a property called fallow, which encapsulates this idea of having somewhere both players want to play using a dead cell.
	
	\begin{definition}\label{def:fallow}
		A game $G$ is called \emph{fallow} if $G \heq H \hs *_d$ for some game $H$. $G$ is called \emph{taut} otherwise.
	\end{definition}
	
	\section{Alternative Characterizations}
	
	We now present many alternative ways to characterize fallow games. We start with the two alternative characterizations that we will use the most.
	
	\begin{theorem}\label{thm:fallow-Gstar}
		A game $G$ is fallow if and only if $G \heq G \hs *_d$.
	\end{theorem}
	
	\begin{proof}
		"$\Rightarrow$" Suppose $G \heq G \hs *_d$. Then $G = H \hs *_d$ for some game $H$, namely $H = G$.
		
		"$\Leftarrow$" Suppose $G \heq H \hs *_d$ for some game $H$. Then we will show that $H \hs *_d \heq H \hs *_d \hs *_d$. By Lemma~\ref{thm:dead-cells-equal} $*_d \hs *_d \heq *_d$. Our desired result follows from transitivity.
	\end{proof}
	
	\begin{theorem}\label{thm:fallow-triangle}
		$G$ is fallow if and only if $G\htri G$.
	\end{theorem}
	
	\begin{proof}
		"$\Rightarrow$" Suppose $G$ is fallow, then $G \heq G \hs \stard$ by Theorem~\ref{thm:fallow-triangle}.  Since $G$ is a right option of $G+*_d$ we have $G\htri G$.
		\\
		"$\Leftarrow$" Suppose $G \htri G$. We will show $G \heq G \hs \stard$. First, take a left option of $G^L$. Then, we have that $G^L$ is a left option of $G \hs \stard$, so $G^L \htri G \hs \stard$. Next, take a right option $(G \hs \stard)^R$ of $G \hs \stard$. There are three possible options.
		\begin{itemize}
			\item [] Case i) $(G \hs \stard)^R = G$. Then $G \htri G$ by assumption.
			\item [] Case ii) $(G \hs \stard)^R = G^R$. Then $G^R \htri G$.
			\item [] Case iii) $(G \hs \stard)^R = G^R \hs \stard$. Then, $G^R$ is a left option of $G^R \hs \stard$ and a right option of $G$. It follows that $G^R \hs \stard \htri G$.
		\end{itemize}
		
		Thus $G \hleq G \hs \stard$. The proof that $G \hs \stard \hleq G$ follows similarly. So we have that $G \heq G \hs \stard$.
	\end{proof}
	
	Now we begin working towards our main results in this chapter, that fallow games have unique ``best" moves, just like $\N$ positions.
	
	\begin{lemma}\label{thm:fallow-gL-to-gRgL}
		For fallow $G$, $G\hleq G^L \Rightarrow G^R \hleq G^{L'}$.
	\end{lemma}
	
	\begin{proof}
		We have $G \htri G \hleq G^L$. It follows that $G \htri G^L$. There are three ways this can happen.
		\begin{itemize}
			\item [] Case i) $G^R \htri G^L$.
			\item [] Case ii) $G^R \htri G^{LL}$. Then, since $G$ is option closed, $G^{LL}$ is a left option of $G$. So, we take $G^{L'} = G^{LL}$ and we are done.
			\item [] Case iii) $G \htri G^{LL}$. Then, by induction hypothesis, it follows that there is some $G^R$ and $G^{L'}$ such that $G^R \hleq G^{L'}$.\qedhere
		\end{itemize}	
	\end{proof}
	
	\begin{lemma}\label{thm:fallow-gR-to-gRgL}
		For fallow $G$, $G^R \hleq G \Rightarrow G^{R'} \hleq G^{L}$.
	\end{lemma}
	
	\begin{proof}
		The proof follow similarly to that for Lemma~\ref{thm:fallow-gL-to-gRgL}
	\end{proof}
	
	\begin{lemma}\label{thm:prev-lemma}
		If $G \htri H + *_d$, then $G \htri H$ or $G \hleq H$.
	\end{lemma}
	
	\begin{proof}
		There are seven possible way we can have $G \htri H + *_d$.
		\begin{itemize}
			\item [] Case i) $G^R \hleq H\hs \stard$. Then we have $G^R \htri H$. It follows from Theorem~\ref{thm:oc-right-closure} that $G \htri H$.
			\item [] Case ii) $G^{R}\hleq H^L\hs \stard$. Then we have $G^{R} \htri H^L$. It follows from Theorems~\ref{thm:oc-right-closure} and \ref{thm:oc-left-closure} that $G \htri H$.
			\item [] Case iii) $G\hleq H^L\hs \stard$. Then we have $G \htri H^L$. It follows from Theorem~\ref{thm:oc-left-closure} that $G \htri H$.
			\item [] Case iv) $G^{R} \hleq H^L$. It follows that $G^{R} \htri H$. Then by Theorem~\ref{thm:oc-right-closure} we have that $G \htri  H$.
			\item [] Case v) $G \hleq H^L$. It follows that $G \htri H$.
			\item [] Case vi) $G^{R} \hleq H$. It follows that $G \htri H$.
			\item [] Case vii) $G \hleq H$.\qedhere
		\end{itemize}
	\end{proof}
	
	\begin{corollary}\label{corollary}
		If $G +* \hleq H + *$, then $G \htri H$ or $G \hleq H$.
	\end{corollary}
	
	\begin{proof}
		$G$ is a left option of $G\hs*_d$. Thus, $G \htri H+*$ and the statement follows immediately from Lemma~\ref{thm:prev-lemma}.
	\end{proof}
	
	\begin{lemma}\label{thm:fallow-lem-taut-equal}
		If $G \hs \stard \hleq H \hs \stard$, $H \hleq G$ and at least one of $G$, or $H$ is taut, then $G \heq H$.
	\end{lemma}
	
	\begin{proof}
		Since $G \hs \stard \hleq H \hs \stard$, we have by Corollary~\ref{corollary} that either $G \htri H$ or $G \hleq H$. But, if we have $G \htri H$, that combined with $H \hleq G$ gives us that $G \htri G$ and $H \htri H$. Since at least one of $G$ and $H$ is taut, this is a contradiction, so we must have that $G \hleq H$. Since $H \hleq G$, it follows that $G \heq H$.
	\end{proof}
	
	\begin{lemma}\label{thm:fallow-taut-options-existence}
		Let $G$ and $H$ be games, where at least one of $G$ or $H$ is fallow and $G \hleq H$. Then, there are taut games $G^{(R)}$ and $H^{(L)}$ such that $G^{(R)} \hleq H^{(L)}$.
	\end{lemma}
	
	\begin{proof}
		Suppose, without loss of generality that $G$ is fallow. Then $G \htri G \hleq H$. It follows from Theorem~\ref{thm:oc-transitivity} that $G \htri H$. There are three ways this can happen.
		\begin{itemize}
			\item [] Case i) $G^R \hleq H$. By induction hypothesis, there are taut $G^{R(R)}$ and $H^{(L)}$ such that $G^{R(R)} \hleq H^{(L)}$. Since $G$ is option closed, we have $G^R \hleq H^{(L)}$, both taut.
			\item [] Case ii) $G \hleq H^L$. By induction hypothesis, there are taut $G^{(R)}$ and $H^{L(L)}$ such that $G^{(R)} \hleq H^{L(L)}$. Since $G$ is option closed, we have $G^{(R)} \hleq H^L$, both taut.
			\item [] Case iii) $G^R \hleq H^L$. By induction hypothesis, there are taut $G^{R(R)}$ and $H^{L(L)}$ such that $G^{R(R)} \hleq H^{L(L)}$. Since $G$ is option closed, we have $G^R \hleq H^L$, both taut.\qedhere
		\end{itemize}
	\end{proof}
	
	\begin{lemma}\label{thm:fallow-add-star-equal}
		For $*$-antimonotone $G$, $G^R \hleq G^L \Rightarrow G^L \hs *_d \heq G^R \hs *_d \heq G$.
	\end{lemma}
	
	\begin{proof}
		Since $G^R \hleq G^L$, it follows that $G^R \hs *_d \hleq G^L \hs *_d$. But since $G$ is $*$-an\-ti\-mo\-no\-tone, we have that $G^L \hs *_d \hleq G \hleq G^R \hs *_d$. It follows that they are all equivalent.
	\end{proof}
	
	\begin{lemma}\label{thm:fallow-star-taut-equal}
		For taut, $*$-antimonotone games $G$ and $H$, $G \hs \stard \heq H \hs \stard$ if and only if $G \heq H$.
	\end{lemma}
	
	\begin{proof}
		"$\Rightarrow$ By Corollary~\ref{corollary} we have that either $G \hleq H$ or $G \htri H$. Similarly, we know that $H \hleq G$ or $H \htri G$. If we can show that either $G \hleq H$ or $H \hleq G$, then by Lemma~\ref{thm:fallow-lem-taut-equal} we would have that $G \heq H$.
		
		Suppose for the sake of contradiction that we have $G \htri H$ and $H \htri G$. There are nine ways this can happen. 
		\begin{itemize}
			\item [] Case i) $G^R \hleq H$ and $H^R \hleq G$. Then, we have that $G^R \hleq H \htri H^R \hleq G$. It follows from Theorem~\ref{thm:oc-transitivity} that $G^R \htri G$. Then, by Theorem~\ref{thm:oc-right-closure} it follows that $G \htri G$. This contradicts the assumption that $G$ was taut, so we do not have $G^R \hleq H$ and $H^R \hleq G$.
			\item [] Case ii) $G^R \hleq H$ and $H \hleq G^L$. Then, we have that $G^R \hleq H \hleq G^L$. It follows from Theorem~\ref{thm:oc-transitivity} that $G^R \hleq G^L$. Thus $G$ is fallow. This contradicts the assumption that $G$ was taut, so we do not have $G^R \hleq H$ and $H \hleq G^L$.
			\item [] Case iii) $G^R \hleq H$ and $H^R \hleq G^L$. Then, we have that $G^R \hleq H \htri H^R \hleq G^L$. It follows from Theorem~\ref{thm:oc-transitivity} that $G^R \htri G^L$. Then, by Theorems~\ref{thm:oc-right-closure}~and~\ref{thm:oc-left-closure} it follows that $G \htri G$. This contradicts the assumption that $G$ was taut, so we do not have $G^R \hleq H$ and $H^R \hleq G^L$.
			\item [] Case iv) $G \hleq H^L$ and $H^R \hleq G$. Then, we have that $H^R \hleq G \hleq H^L$. It follows from Theorem~\ref{thm:oc-transitivity} that $H^R \hleq H^L$. Thus, $H$ is fallow. This contradicts the assumption that $H$ was taut, so we do not have $G \hleq H^L$ and $H^R \hleq G$.
			\item [] Case v) $G \hleq H^L$ and $H \hleq G^L$. Then, we have that $G \hleq H^L \htri H \hleq G^L$. It follows from Theorem~\ref{thm:oc-transitivity} that $G \htri G^L$. Then, by Theorem~\ref{thm:oc-left-closure} it follows that $G \htri G$. This contradicts the assumption that $G$ was taut, so we do not have $G \hleq H^L$ and $H \hleq G^L$.
			\item [] Case vi) $G \hleq H^L$ and $H^R \hleq G^L$. Then, we have that $H^R \hleq G^L \htri G \hleq H^L$. It follows from Theorem~\ref{thm:oc-transitivity} that $H^R \htri H^L$. Then, by Theorems~\ref{thm:oc-right-closure}~and~\ref{thm:oc-left-closure} it follows that $H \htri H$. This contradicts the assumption that $H$ was taut, so we do not have $G \hleq H^L$ and $H^R \hleq G^L$.
			\item [] Case vii) $G^R \hleq H^L$ and $H^R \hleq G$. Then, we have that $H^R \hleq G \htri G^R \hleq H^L$. It follows from Theorem~\ref{thm:oc-transitivity} that $H^R \htri H^L$. Then, by Theorems~\ref{thm:oc-right-closure}~and~\ref{thm:oc-left-closure} it follows that $H \htri H$. This contradicts the assumption that $H$ was taut, so we do not have $G^R \hleq H^L$ and $H^R \hleq G$.
			\item [] Case viii) $G^R \hleq H^L$ and $H \hleq G^L$. Then, we have that $G^R \hleq H^L \htri H \hleq G^L$. It follows from Theorem~\ref{thm:oc-transitivity} that $G^R \htri G^L$. Then, by Theorems~\ref{thm:oc-right-closure}~and~\ref{thm:oc-left-closure} it follows that $G \htri G$. This contradicts the assumption that $G$ was taut, so we do not have $G^R \hleq H^L$ and $H \hleq G^L$.
			\item [] Case ix) $G^R \hleq H^L$ and $H^R \hleq G^L$. Then we have that $G^R \hs \stard \hleq H^L \hs \stard$ and $H^R \hs \stard \hleq G^L \hs \stard$. But by $*$-antimonotonicity, we have that $H^L \hs \stard \hleq H \hleq H^R \hs \stard$ and $G^L \hs \stard \hleq G \hleq G^R \hs \stard$.
			
			It follows that $G^R \hs \stard \heq H^L \hs \stard \heq H^R \hs \stard \heq G^L \hs \stard$. Thus, by induction hypothesis we have that $G^R = G^L$. This contradicts the assumption that $G$ was taut, so we do not have $G^R \hleq H^L$ and $H^R \hleq G^L$.
		\end{itemize}
		So we do not have that $G \htri H$ and $H \htri G$. Thus, we must have that either $G \hleq H$ or $H \htri G$. It follows that $G \heq H$.
		
		"$\Leftarrow$" This follows immediately from Theorem~\ref{thm:order-sums}.
	\end{proof}
	
	\begin{theorem}\label{thm:*antimonotone-unique-taut-positions}
		For option closed, $*$-antimonotone $G$, $G$ is fallow if and only if there exist unique $G^L, G^R$ both taut, such that $G^R \heq  G^L$.
	\end{theorem}
	
	\begin{proof}
		First we show such options must exist. Since $G$ is fallow, we have that $G \htri G$. There are three ways this can happen.
		\begin{itemize}
			\item [] Case i) $G \hleq G^L$. Then we have by Lemma~\ref{thm:fallow-gL-to-gRgL} that there exist $G^R$ and $G^{L'}$ such that $G^R \hleq G^{L'}$
			\item [] Case ii) $G^R \hleq G$. Then we have by Lemma~\ref{thm:fallow-gR-to-gRgL} that there exist $G^{R'}$ and $G^L$ such that $G^{R'} \hleq G^{L}$.
			\item [] Case iii) $G^R \hleq G^L$.
		\end{itemize}
		
		So we have $G^R \hleq G^L$ for some $G^R$ and $G^L$. By Lemma\ref{thm:fallow-taut-options-existence}, there are $G^{R(R)} $ and $G^{L(L)}$ which are taut such that $G^{R(R)} \hleq G^{L(L)}$. Those are right and left options of $G$, since $G$ is option closed. So we have taut $G^{R'}$ and $G^{L'}$ such that $G^{R'} \hleq G^{L'}$. It follows from Lemma~\ref{thm:fallow-add-star-equal} that $G^{L'} \hs *_d \hleq G^{R'} \hs *_d$. Then, by Lemma~\ref{thm:fallow-lem-taut-equal} that $G^{R'} \heq G^{L'}$.
		
		Thus we have taut $G^L$ and $G^R$ such that $G^R \heq G^L$. We now show that these $G^L$ and $G^R$ are unique. Suppose we have $G^L \heq G^R$ and $G^{L'} \heq G^{R'}$. Then we have that $G^L \hs \stard \heq G^R \hs \stard \heq G^{L'} \hs \stard \heq G^{R'} \hs \stard \heq G$ by Lemma~\ref{thm:fallow-add-star-equal}. Thus, by Lemma~\ref{thm:fallow-star-taut-equal} we have that $G^L \heq G^{L'}$ and $G^R \heq G^{R'}$.
	\end{proof}
	
	\begin{theorem}
		For a fallow Rex+ game $G$, Left and Right play the same set of cells X to achieve $G^L$ and $G^R$ respectively, where $G^L$, $G^R$ are the unique taut $G^L$, and $G^R$ such that $G^R \heq  G^L$.
	\end{theorem}
	
	\begin{proof}
		Let $X$ be the set of cells that Black plays in to achieve $G^L$, and $Y$ the set of cells that White plays in to achieve $G^R$. Suppose that $X \neq Y$. Then, without loss of generality, there are cells in $Y$ which are not in $X$. 
		
		Let $G^{LR}$ be the move where White fills all the cells in $Y$ which are not in $X$. Then, let $G^{R(L)}$ be the move where Black fills all the cells in $X$ not in $Y$. If there are no such cells, then we take $G^{R(L)} = G^R$.
		
		Then, $G^{LR}$ and $G^{R(L)}$ are identical, except that potentially some stones that were black in $G^{LR}$ became white in $G^{R(L)}$. Since this is never bad for Black, we have that $G^{LR} \hleq G^{R(L)}$. It follows that $G^L \htri G^R$. But then we would have that $G^L \htri G^L$, since $G^L \heq G^R$. This contradicts the assumption that $G^L$ was taut, so we must have that $X = Y$.
	\end{proof}
	
	We call these cells fallow cells, and they are an attempt to generalize dead cells. This also generalizes our observation in Chapter~\ref{chap:examples} that $\N$ positions have unique winning moves and that they are the same for both players. We conjecture that fallow cells are the only cells ever where both players want to play, and then these unique taut positions guaranteed by Theorem~\ref{thm:*antimonotone-unique-taut-positions} are the fallow positions with all the fallow cells filled in. If fallow cells are truly the only places that are good for both players to play in, then these taut position would be the best move for both players, with a unique set of cells to play in to achieve it.

\chapter{Conclusion and Future Work}

In this thesis we developed an intrinsic and a contextual order on option closed games over posets. We defined a class of games under which this order is transitive, and has unique canonical forms. We also defined a new property of some Rex+ games called fallowness, and provided some alternative characterizations of fallowness.

In the future we hope to find alternative characterizations of the classes $\C$ and $\D$ that are easier to work with then the current definition. We also wish to find more ways to characterize taut games. Finally, we wish to prove that the contextual order implies our intrinsic order.

\bibliographystyle{plain}
\bibliography{thesis}

\end{document}